\documentclass[12pt]{article}
\usepackage[utf8]{inputenc}
\usepackage[english]{babel}
\usepackage{amsmath, amssymb}
\usepackage{todonotes}
\usepackage{amsthm}
\usepackage{mathtools}
\usepackage{enumitem}
\usepackage[pagestyles,raggedright]{titlesec}
\usepackage[square,numbers]{natbib}
\usepackage{tocloft}

\usepackage{titlesec}
\titleformat*{\section}{\bfseries\fontsize{15}{20}\selectfont}
\titleformat*{\subsection}{\bfseries\fontsize{13}{16}\selectfont}
\usepackage[a4paper,
            left=2cm,right=2cm,top=2cm,bottom=2cm,]{geometry}
\theoremstyle{definition}
\newtheorem{theorem}{Theorem}[subsection]
\newtheorem{lemma}[theorem]{Lemma}
\newtheorem*{theorem*}{Theorem}

\newtheorem{prop}[theorem]{Proposition}

\newtheorem{definition}[theorem]{Definition}
\newtheorem{remark}[theorem]{Remark}
\newtheorem{thm}{Theorem}[section]
\newtheorem{lem}[thm]{Lemma}
\newtheorem{pro}[thm]{Proposition}
\newtheorem{defn}[thm]{Definition}
\newtheorem{rem}[thm]{Remark}
\newtheorem{cor}[thm]{Corollary}
\newtheorem{lemmat}{Lemma}
\usepackage[hidelinks,colorlinks=true,linkcolor=blue,citecolor=blue]{hyperref}

\begin{document}

\begin{titlepage}
   \begin{center}
       \vspace*{1cm}

       \huge\textbf{Aschbacher's Theorem for the General Linear Group}

       \vspace{0.5cm}

       \vspace{1.5cm}

       \normalsize\textbf{Kadeem Harrigan}

       \vfill

       This is a MSc thesis submitted at Imperial College London in\\
         September 2021, under the supervision of Professor Martin Liebeck.

       \vspace{0.8cm}

   \end{center}
\end{titlepage}

\clearpage

\pagenumbering{roman}
\begin{center}

\Large\textbf{Abstract}
    
\end{center} 

In 1984, Michael Aschbacher proved a seminal classification theorem for the maximal subgroups of effectively all of the classical groups. In this thesis we give a comprehensive, yet accessible description and proof of Aschbacher's theorem, restricting its scope to the general linear group. The main theorem of this paper classifies the maximal subgroups of the general linear group into nine different classes; eight of which have natural descriptions based on an object that their members act on and stabilise, whilst the ninth class - though not having such a natural description - contains groups that are bound by the property of having a unique normal quasisimple subgroup that acts absolutely irreducibly on the vector space. We give a detailed description of each of the first eight classes before proving that if a subgroup is not contained in a member of one of them, then it must have the properties that make up the ninth class. This paper uses techniques that cross over the fields of group theory, linear algebra and representation theory and it is approachable for anyone with an undergraduate understanding of these subjects.

\clearpage
\tableofcontents
\addtocontents{toc}{\vspace{1cm}}

\clearpage\section*{Preface}
In 1984, Michael Aschbacher proved a seminal result (\cite{aschbacher_maximal_1984}) which classifies the maximal subgroups of a large number of groups, many of which are finite simple, known collectively as the classical groups. We will give a precise definition of these classical groups in the first section of our paper, but roughly speaking, they are a collection of groups of linear maps (and their quotient groups) associated with six specific families; the linear groups \(GL(V)\), the symplectic groups \(Sp(V)\), the unitary groups \(GU(V)\) and three types of orthogonal groups \(O^{\circ}(V)\), \(O^{+}(V)\) and \(O^{-}(V)\). Aschbacher's theorem states that, given one of these classical groups, all of its subgroups are contained in a member of one of eight classes - known as Aschbacher classes - or in a member of an `anomaly' class. 

Each of the eight Aschbacher classes have a natural description based on an object that their members act on and stabilise; they are roughly described as follows.
\begin{enumerate}[label=(\roman*),ref=(\roman*)]

\item The \(\mathcal{C}_1\) class contains subgroups that stabilise a subspace of \(V\).
\item The \(\mathcal{C}_2\) class contains subgroups that stabilise a direct sum decomposition of \(V\).
\item The \(\mathcal{C}_3\) class contains subgroups that stabilise an extension field of the underlying field.
\item The \(\mathcal{C}_4\) class contains subgroups that stabilise a tensor product decomposition, with non-isometric factors, of \(V\).
\item The \(\mathcal{C}_5\) class contains subgroups that stabilise a subfield of the underlying field.
\item The \(\mathcal{C}_6\) class contains subgroups that normalise a symplectic type \(r\)-group (\(r\) a prime) embedded in the classical group.
\item The \(\mathcal{C}_7\) class contains subgroups that stabilise a tensor product decomposition, with isometric factors, of \(V\).
\item The \(\mathcal{C}_8\) class contains subgroups that stabilise a particular form on \(V\).
\end{enumerate}
The anomaly class is set apart from the rest in that it doesn't have such a natural description. The force of Aschbacher's theorem is that the eight Aschbacher classes are easily described and well understood and this anomaly class has very restrictive properties which are outlined more fully in points (i) and (ii) in the statement of the theorem below.

Owing to the breadth of its application to effectively all classical groups, Aschbacher's proof is loaded with complex notation that is surely appropriate for stating and proving results in such generality. Indeed, Aschbacher's ability to utilise and create notation and techniques that were adequate for the scope of this task emphasises his ingenuity. However, the notational complexity makes it quite difficult for the uninitiated reader to understand and appreciate the beauty of the theorem and proof. The aim of this current paper then, is to prove the theorem stated below, which is Aschbacher's theorem restricted to just one of the families of classical groups; \(GL(V)\). This restriction allows us to achieve the aim of providing a comprehensive and detailed introduction to Aschbacher's theorem and proof, whilst remaining highly accessible and notationally light. Indeed, our work should be approachable for anyone who has a undergraduate-level understanding of group theory, linear algebra and representation theory. The main theorem that we prove is as follows.

\begin{theorem*}(Aschbacher's Theorem for the General Linear Group)\\
Let \(F\) be a finite field and let \(V\) be a \(n\)-dimensional \(F\)-vector space, for some positive integer \(n\). If \(H\) is a subgroup of \(GL(V,F)\), not containing \(SL(V,F)\), then \(H\) is either contained in a member of one of the Aschbacher classes \(\mathcal{C}_1 - \mathcal{C}_8\) or the following hold.
\begin{enumerate}[label=(\roman*),ref=(\roman*)]
 \item \(H\) has a unique normal quasisimple subgroup \(L\).
 \item \(V\) is an absolutely irreducible \(FL\)-module that cannot be realised over any proper subfield of \(F\) and \(L\) does not fix any classical form on \(V\).
 \end{enumerate}
\end{theorem*}

Some terminology in this definition may be unfamiliar to the reader. We give some brief descriptions here, whilst saving precise definitions for the succeeding sections.
\begin{itemize}
  \item[] A group \(G\) is \textit{quasisimple} if it is perfect and \(G/Z(G)\) is simple. See \hyperlink{quasi}{(0.2.19)}.
  
  \item[] An irreducible \(FG\)-module is said to be \textit{absolutely irreducible} if it remains irreducible when we extend the scalars of the field to encompass a larger field. See \hyperlink{0.3.9}{(0.3.9)} and \hyperlink{0.3.10}{(0.3.10)}.
  
  \item[]A \(FG\)-module \(V\) is said to be \textit{realised} over a subfield \(k\subset F\), if there is a basis of \(V\) such that the \(k\)-span of that basis is stabilised by the \(G\)-action. See \hyperlink{5.4}{(5.4)}.
  \item[]The \textit{classical forms} are a specific collection of sesquilinear and quadratic forms that give rise to the classical groups. A \(FG\)-module is said to \textit{fix} a classical form if that form is \(G\)-invariant. See \hyperlink{0.1.8}{(0.1.8)} and \hyperlink{0.1.18}{(0.1.18)}.
  
\end{itemize}
The reader may also note that we have excluded the overgroups of \(SL(V,F)\) from the classification. This is due to the fact that \(GL(V,F)/SL(V,F)\cong F^{\times}\) and so such overgroups correspond to subgroups of \(F^{\times}\). 

We consider this a somewhat trivial case as such a group is, modulo \(SL(V,F)\), a group of scalars. Thus, in effect, all we are doing is setting aside the well-known \(SL(V,F)\). 

\noindent Having now the general gist of what we are setting out to prove, the specifics are unfolded in the rest of our paper as follows.

We begin with a section of preliminary results where we recall and introduce some fundamental propositions, lemmas and definitions from the three main disciplines of algebra that this theorem crosses: linear algebra, group theory and modules and representation theory. In this section we will give a formal definition of the classical forms, the classical groups, quasisimple groups and absolute irreducibility.

After this, in sections one to eight, we will give full and formal descriptions of the Aschbacher classes one to eight. In each section we define the class, explain its group structure and prove any related results necessary for the proof of the main theorem. We will also comment on any divergence between our definition of that class and Aschbacher's original definition - we do this to give assistance to the reader who wishes to use our paper as a stepping stone to understanding the fullness of Aschbacher's 1984 masterpiece. On this point, it should be noted that Aschbacher's original definitions of the eight classes were not intended to avoid overlap between one another. We however, have followed the convention of \cite{kleidman_subgroup_1990} in erasing such overlap between classes and in each section we discuss when and how we have done this.

Our final section, nine, is where we prove the main theorem. The proof is broken down into eight individual lemmas, each lemma corresponds to one of the Aschbacher classes and makes use of the results stated in the section corresponding to that class. Therefore, after the section of preliminary results (which we have accordingly titled Section \(0\)) our paper has the following correspondence between sections.
\[ \text{Aschbacher class } \mathcal{C}_i \xleftrightarrow{} \text{Section \(i\)} \xleftrightarrow{} \text{ Lemma \(i\) in the main proof}\]

Throughout this paper, the groups and fields are always finite and the vector spaces are always finite dimensional. Unless otherwise stated, \(F\) is a field of size \(q=p^e\), for some prime \(p\) and positive integer \(e\) and \(V\) is a \(n\)-dimensional \(F\)-vector space, for some positive integer \(n\). The group \(GL(V,F)\) will often be abbreviated \(GL(V)\) and often used interchangeably with the matrix group \(GL_n(q)\) where appropriate and useful. We will also regularly make use of the convention of denoting the group of scalars inside \(GL(V,F)\) by \(F^{\times}\). 

Besides \hyperlink{4.3}{(4.3)}, the proofs given in this paper are an original presentation, although many are inspired by and adapted from the references in the bibliography - particularly \cite{aschbacher_maximal_1984} and \cite{kleidman_subgroup_1990}.

\clearpage

\newpagestyle{main}{%
  \sethead[\thepage][][\chaptername\ \thechapter. \chaptertitle]{\thesection\ \sectiontitle}{}{\thepage}
  \headrule
}
\pagestyle{main}

\section{Preliminary Results}
\pagenumbering{arabic}
\setcounter{page}{1}

In this section, we cover the preliminary results and definitions that form the groundwork for the rest the paper. Proofs of lemmas and propositions are omitted, but references are provided.

\subsection{Linear Algebra}

Here we discuss the notion of forms on vector spaces,  with a view towards defining the \textit{classical forms}. 

\begin{definition}

A map \(f: V\times V\longrightarrow F\) is called a \textit{sesquilinear form} if there exists \(\theta\in Aut(F)\) such that for all \(v,w,z\in V\) and \(\lambda \in F\), the following hold.

\begin{enumerate}[label=(\roman*),ref=(\roman*)]
\item \(f(v+w,z)=f(v,z)+f(w,z)\).
\item \(f(v,w+z)=f(v,w)+f(v,z)\).
\item \(f(\lambda v,w)=\lambda f(v,w)\).
\item \(f(v,\lambda w)= \lambda^{\theta} f(v,w)\).
\end{enumerate}

We call \(f\) a \textit{bilinear form} if the above holds and \(\theta\) is the identity map. 
\end{definition}

\begin{definition}
A map \(Q:V\longrightarrow F\) is called a \textit{quadratic form} if for all \(v\in V\) and \(\lambda \in F\) the following hold.

\begin{enumerate}[label=(\roman*),ref=(\roman*)]
\item \(Q(\lambda v)=\lambda^2Q(v)\)
\item The map \(f_Q(v,w):=Q(v+w)-Q(v)-Q(w)\) is a bilinear form
\end{enumerate}
The map \(f_Q\) is referred to as the \textit{associated bilinear form} of \(Q\).
\end{definition}

Given a collection of these forms, we may define another such form on a tensor product space as follows.

\begin{definition}
\hypertarget{0.1.3}{Let} \(f_1\) and \(f_2\) be sesquilinear forms on \(F\)-vector spaces \(V_1\) and \(V_2\), both of which being associated with the same \(\theta \in Aut(F)\). If \(\{v_1,...\,,v_m\}\) and \(\{w_1,...\,,w_k\}\) are bases for \(V_1\) and \(V_2\), then we can define a sesquilinear form \(f_1\otimes f_2\) on \(V_1\otimes V_2\) by: \((f_1\otimes f_2)(v_i\otimes w_j, v_{i'}\otimes w_{j'})=f_1(v_i,v_{i'})f_2(w_j,w_{j'})\), which is extended to the whole of \(V_1\otimes V_2\) by \(F\)-linearity.
\end{definition}

\begin{remark}
This construction can be extended to any number of sesquilinear forms (and tensor product factors) and the same definition can be made for a collection of quadratic forms.

\end{remark}

The following definitions give additional description for sesquilinear and quadratic forms.

\begin{definition}
A sesquilinear form \(f\) is called \textit{non-degenerate} if there are no non-zero vectors \(v\in V\) such that \(f(v,w)=f(w,v)=0\), for all \(w\in V\). A quadratic form \(Q\) is said to be \textit{non-degenerate} if its associated bilinear form is. 
\end{definition}

\begin{definition}
We say that a sesquilinear form \(f\) is:

\begin{enumerate}[label=(\roman*),ref=(\roman*)]
\item \textit{symmetric} if it is bilinear and \(f(v,w)=f(w,v)\), for all \(v,w\in V\).
\item \textit{skew-symmetric} if it is bilinear and \(f(v,w)=-f(w,v)\), for all \(v,w\in V\).
\item \textit{conjugate-symmetric} if the order of the field is a square and for all \(v,w\in V\), we have \(f(v,w)=f(w,v)^{\theta}\), where \(\theta\) is a field automorphism of order two.
\item \textit{alternating} if it is bilinear and \(f(v,v)=0\), for all \(v\in V\).
\end{enumerate}

\end{definition}

\begin{remark}

\hypertarget{0.1.7}{These} four definitions are not mutually exclusive. Indeed, every alternating form is skew-symmetric, since if \(f\) is alternating, then for all \(v,w\in V\):
\[f(v,w)+f(w,v)=f(v,v)+f(v,w)+f(w,v)+f(w,w)=f(v+w,v+w)=0\]
We have further overlap of these definitions that is dependent on the characteristic of \(F\). If \(F=\mathbb{F}_q\) and \(q\) is even, then a skew-symmetric form is evidently the same as a symmetric form. If \(q\) is odd, then a skew-symmetric form is the same as an alternating form, since \(f(v,v)=-f(v,v)\) must imply that \(f(v,v)=0\).  

The characteristic of \(F\) also impacts the associated bilinear form of a quadratic form \(Q\). Indeed, we have seen that \(Q\) induces \(f_Q\). If in addition, \(q\) is odd, then \(f_Q\) also induces \(Q\) by the formula \(Q(v)=\tfrac{1}{2}f_Q(v,v)\). On the other hand, if \(q\) is even we cannot, in general, recover a quadratic form from a bilinear form. 

By definition, an associated bilinear form is always symmetric. However, if we restrict to the case where \(q\) is even, we can also show that the associated form must be alternating, since:
\[f_Q(v,v)=Q(2v)-Q(v)-Q(v)=4Q(v)-2Q(v)=0\]
\end{remark}

We are now ready define the classical forms.

\begin{definition}
Let \(V\) be a \(n\)-dimensional \(F\)-vector space. \hypertarget{0.1.8}{A} sesquilinear or quadratic form on \(V\) will be referred to as a \textit{classical form} if it is one of the following.

\begin{enumerate}[label=(\roman*),ref=(\roman*)]
\item The \textit{zero form}; \(f:V\times V\longrightarrow F;\,(v,w)\mapsto 0\), for all \(v,w\in V\).
\item A non-degenerate alternating bilinear form which we refer to as a \textit{symplectic form}.
\item A non-degenerate conjugate-symmetric sesquilinear form, which we refer to as a \textit{unitary form}.
\item A non-degenerate symmetric bilinear form over a field of odd characteristic.
\item A non-degenerate quadratic form, which we refer to as an \textit{orthogonal form}.
\end{enumerate}

\end{definition}

\noindent Next, we define two different notions of equivalences between forms.

\begin{definition}
Let \(f\) and \(f'\) be two sesquilinear forms on \(V\) and \(g\in GL(V)\).
\begin{enumerate}[label=(\roman*),ref=(\roman*)]
\item We say that \(g\) is an \textit{isometry} between \(f\) and \(f'\) if \(f(g(v),g(w))=f'(v,w)\), for all \(v,w\in V\).
\item We say that \(g\) is a \textit{similarity} between \(f\) and \(f'\) if there exists a \(\lambda_g\in F\) such that \(f(g(v),g(w))=\lambda_gf'(v,w)\), for all \(v,w\in V\).
\end{enumerate}
If there exists such an isometry or similarity, we say that \(f\) and \(f'\) are \textit{isometric} or \textit{similar}. 
\end{definition}

\begin{remark}
The same definitions can be ascribed to a quadratic form \(Q\), where the isometry or similarity is just acting on the single argument of the form.
\end{remark}

\begin{definition}
If \(f=f'\) (or \(Q=Q'\)) in the definition above, then we say that \(g\) is an \textit{isometry/similarity of the form \(f\) (or \(Q\))}. The set of all isometries of \(f\) (or \(Q\)) is denoted \(I(V,f)\) (or \(I(V,Q)\)) and the set of all similarities of \(f\) is denoted \(\Delta(V,f)\) (or \(\Delta(V,Q)\)).
\end{definition}

\begin{remark}
\hypertarget{0.1.12}{Let} \(Q\) be an orthogonal form. In view of our discussion in \hyperlink{0.1.7}{(0.1.7)}, if \(q\) is even, then \(I(V,Q)\subset I(V,f_Q)\), where \(f_Q\) is a symplectic form. If \(q\) is odd, then \(I(V,Q)=I(V,f_Q)\).
\end{remark}

We now state an important classification result for classical forms; which tells us the conditions satisfied by \(V\) and \(F\) for a particular form to exist, and when such forms do exist, how many there are up to isometry or similarity. The result is proved by showing that there are very restricted conditions on the basis of a vector space with a given form attached to it. Thus, the classification also provides information about these bases - we include this information in the following statement only to the degree that it is relevant for our discussions. We commend the reader to (\cite{kleidman_subgroup_1990}, p.22-28) for a more detailed treatment of this classification, including a proof.

\begin{lemma}
\hypertarget{0.1.13}{We} have the following classification of classical forms.
\begin{enumerate}[label=(\roman*),ref=(\roman*)]
\item A symplectic form exists on \(V\)  if and only if \(n\) is even. Moreover, this form is unique up to isometry. 

\item A unitary form exists on \(V\) if and only if \(F=\mathbb{F}_{q}\), where \(q\) is a square. Moreover, this form is unique up to isometry and it admits an orthonormal basis of \(V\).

\item \hypertarget{class}{If} \(n=2m+1\) is odd, there exists a unique orthogonal form up to similarity. Forms in this similarity class are referred to as orthogonal forms of \textit{\(\circ\)-type}.

\item If \(n=2m\) is even, there are two orthogonal forms (up to isometry) that can be distinguished by the dimension of the maximal subspace of \(V\) on which the form is uniformly zero. If this maximal subspace is of dimension \(m\), we say that the orthogonal form is of \textit{plus-type}. If the subspace has dimension \(m-1\), we say that the orthogonal form is of \textit{minus-type}.
A plus-type orthogonal form \(Q\) admits a basis \(\{x_1,...\,,x_m,y_1,...\,,y_m\}\), such that \(Q(x_i)=Q(y_j)=0\) and \(f_Q(x_i,y_j)=\delta_{ij}\), for all \(1\leq i,j \leq m\). A minus-type orthogonal form admits an orthonormal basis when \(q\equiv 3\:(mod\:4)\) and \(m\) is odd. 

\end{enumerate}
\end{lemma}

\begin{remark} In part (iii), the equivalence condition is up to similarity (as opposed to isometry as in the other cases). However, we shall see shortly that the set of isometries of these forms have the structure of a group, and it is these isometry groups that we are chiefly concerned with. With this in mind, we note that it follows from the definition that similar forms have isomorphic isometry groups. 

\end{remark}

We conclude our discussion on classical forms with an alternative way of understanding isometries of forms, using matrices. The following definition shows how we may associate particular matrices to a given form.

\begin{definition}
If \(\mathcal{B}=\{v_1,...\,,v_n\}\) is a basis for \(V\) and \(f\) is a sesquilinear form on \(V\), we define the \textit{matrix of f with respect to \(\mathcal{B}\)} as \(B=(a_{ij})\), where \(a_{ij}:=f(v_i,v_j)\).
\end{definition}

Next we define an action of a field automorphism on a matrix over that field.

\begin{definition}
Let \(F=\mathbb{F}_q\), \(A=(a_{ij})\in GL_n(q)\) and \(\theta\in Aut(F)\). We define \(A^{\theta}\) to be the matrix \((b_{ij})\in GL_n(q)\), where \(b_{ij}=\theta(a_{ij})\).
\end{definition}

Recall the following result about field automorphisms. 
 
\begin{prop}
\hypertarget{0.1.17}{If} \(F=\mathbb{F}_{p^e}\), then \(Aut(F)=\{x\mapsto x^{p^j}\,|\,0\leq j \leq e-1\}\cong C_e\).
\end{prop}

From this result we can deduce that if \(F=\mathbb{F}_q\), then there is an order-two automorphism if and only if \(q\) is a square. Furthermore, when it does exist, it is the unique automorphism of its order. This fact allows us to state the next proposition, which follows almost immediately from the definitions of isometries and matrices of forms. See (\cite{bray_maximal_2013}, p.17) for a proof.

\begin{prop}
\hypertarget{0.1.18}{Let} \(V\) be a \(\mathbb{F}_q\)-vector space with classical form \(f\). If \(\mathcal{B}\) is a basis for \(V\) and \(B\) is the matrix of \(f\) with respect to \(\mathcal{B}\), then the following hold.
\begin{enumerate}[label=(\roman*),ref=(\roman*)]

\item \(I(V,f)\cong \{A\in GL_n(q)\,|\, ABA^t=B\}\), when \(f\) is a symplectic or non-degenerate symmetric bilinear form or \(f\) is the zero form.

\item \(I(V,f)\cong \{A\in GL_n(q)\,|\, ABA^{t\theta}=B\}\), when \(f\) is a unitary form and \(\theta\) is the field automorphism of order two.

\item  If \(Q\) is an orthogonal form and \(f=f_Q\) is the associated symmetric bilinear form, then \(I(V,Q)\cong \{g\in I(V,f_Q)\,|\, Q(g(v))=Q(v) \text{, for all } v\in V\}\).
\end{enumerate}
\end{prop}

A subgroup of \(GL_n(q)\) is said to \textit{fix} a classical form \(f\) or \(Q\) if it is contained in the set corresponding to that form in the lemma above. Similar can be said of a subgroup of \(GL(V)\) if, given an arbitrary basis of \(V\), the above definition holds for the corresponding matrix group.

\begin{remark}
\hypertarget{0.1.19}{We} note the following observations.
\begin{enumerate}[label=(\roman*),ref=(\roman*)]
\item If \(f\) is the zero form then \(B\) is the zero matrix and therefore \(I(V,f)\cong GL_n(q)\).

\item If \(f\) is a unitary form, then \hyperlink{0.1.13}{(0.1.13.ii)} tells us that there is a basis \(\mathcal{B}\) such that the matrix of \(f\) with respect to \(\mathcal{B}\) is \(I_n\) and therefore \(I(V,f)\) is isomorphic to the group of unitary matrix.

\item If \(Q\) is a orthogonal form, then \hyperlink{0.1.13}{(0.1.13)} tells us that in some, but not all, cases there exists a basis \(\mathcal{B}\) such that the matrix of \(Q\) with respect to \(\mathcal{B}\) is \(I_n\). In these cases, \(I(V,f_Q)\) is isomorphic to the group of orthogonal matrices.

\end{enumerate}
\end{remark}

\subsection{Group Theory}

We begin this subsection by defining two groups which are generalisations of the well-known dihedral and quaternion groups - we will be referring to these in section six.

\begin{definition}
Let \(n\geq 4\). The \textit{semidihedral group} is the group of order \(2^n\) which can be presented \(SD_{2^n}=\langle x,y\,|\,x^{2^{n-1}}=y^2=1,\,yxy=x^{2^{n-2}-1}\rangle\).
\end{definition}

\begin{definition}
Let \(n\geq 2\). The \textit{generalised quaternion group} is the group of order \(2^{n+1}\) which can be presented \(Q_{2^{n+1}}=\langle i,j\,|\,i^{2^n}=1,\,i^{2^{n-1}}=j^2\,j^{-1}ij=i^{-1}\rangle\).
\end{definition}

\subsubsection*{Group products and extensions}

We now define some group products and extensions which we use throughout our paper.

\begin{definition}
Let \(H\) and \(K\) be groups with a homomorphism \(\phi: K\longrightarrow Aut(H)\) (or equivalently, \(\phi\) is an action of \(K\) on \(H\)). We define the \textit{semidirect product} of \(H\) and \(K\) with respect to \(\phi\), denoted \(H\rtimes K\), to be the group with underlying set \(H\times K\) and group operation:
\[(h',k')\cdot (h,k)=(h'h^{\phi(k')},k'k)\]
\end{definition}

\begin{definition}
\hypertarget{0.2.4}{Let} \(H\) and \(K\) be groups and let \(\phi:K\longrightarrow S_m\) be a homomorphism. The \textit{wreath product} of \(H\) by \(K\) with respect to \(\phi\), denoted \(H\wr K\), is the semidirect product \(H^m\rtimes K\), where \(K\) acts on \(H^m\) by permuting the coordinates via \(\phi\).
\end{definition}

\begin{definition}
Let \(H\) and \(K\) be groups with central subgroups \(H_1\) and \(K_1\) such that there exists an isomorphism \(\phi:H_1\longrightarrow K_1\). An \textit{external central product} of \(H\) and \(K\), denoted \(H\circ K\), is a quotient of the group \(H\times K\) by the subgroup \(Z=\{(h,k)\,|\,h\in H_1,\, k\in K_1,\, \phi(h)=k^{-1}\}\).
A group \(G\) is said to be an \textit{internal central product} of \(H\) and \(K\), if \(H\) and \(K\) are subgroups of \(G\) such that \(G=HK\) and \(H\) and \(K\) commute with each other.
\end{definition}

\begin{definition}
Let \(G\), \(H\) and \(K\) be groups. We say that \(G\) is an \textit{extension} of \(K\) by \(H\) if there exists a surjective homomorphism \(\beta:G\longrightarrow K\) and an injective homomorphism \(\alpha:H\longrightarrow G\) such that \(\alpha(H)\unlhd G\) and \(im\alpha=ker\beta\). 
\end{definition}

Direct and semidirect products are examples of extensions; however, there are many other extensions that don't fall in into these two categories. We will use the notation \(H.K\) for an unspecified extension of \(H\) by \(K\).

\subsubsection*{Classical groups}

The classification of classical forms enables us to define the \textit{classical groups}, which we build up to in the following three definitions.

\begin{definition}
Let \(F=\mathbb{F}_q\) and let \(V\) be a \(n\)-dimensional \(F\)-vector space. \hypertarget{0.2.7}{If} \(f\) or \(Q\) is a classical form on \(V\), then \(I(V,f)\) and \(I(V,Q)\) have the structure of a group (with respect to composition) and are named and denoted as follows.
\begin{enumerate}[label=(\roman*),ref=(\roman*)]

\item If \(f\) is the zero form, then as stated in \hyperlink{0.1.19}{(0.1.19.i)}, \(I(V,f)\) is the general linear group \(GL(V)\).

\item If \(n=2m\) is even and \(f\) is a symplectic form, then \(I(V,f)\) is called the \textit{symplectic group} of \(f\), denoted \(Sp(V)\). Given an arbitrary basis, the corresponding matrix group is denoted \(Sp_{2m}(q)\).

\item If \(q\) is a square and \(f\) is a unitary form, then \(I(V,f)\) is called the \textit{unitary group} of \(f\), denoted \(GU(V)\). Given an arbitrary basis, the corresponding matrix group is denoted \(GU_{n}(q^{1/2})\).

\item  If \(n=2m+1\) is odd and \(Q\) is an orthogonal form, then \(I(V,Q)\) is called the \textit{\(\circ\)-type orthogonal group} of \(Q\), denoted \(O^{\circ}(V)\). Given an arbitrary basis, the corresponding matrix group is denoted \(O^{\circ}_{2m+1}(q)\).

\item  If \(n=2m\) is even and \(Q\) is an orthogonal form of plus-type, then \(I(V,Q)\) is called the \textit{plus-type orthogonal group} of \(Q\), denoted \(O^{+}(V)\). Given an arbitrary basis, the corresponding matrix group is denoted \(O^{+}_{2m}(q)\).

\item  If \(n=2m\) is even and \(Q\) is an orthogonal form of minus-type, then \(I(V,Q)\) is called the \textit{minus-type orthogonal group} of \(Q\), denoted \(O^{-}(V)\). Given an arbitrary basis, the corresponding matrix group is denoted \(O^{-}_{2m}(q)\).

\end{enumerate}
\end{definition}

\begin{definition}
We define some important subgroups, overgroups and quotient groups of those defined above.
\begin{enumerate}[label=(\roman*),ref=(\roman*)] 
\item We define \(S(V,f)\) and \(S(V,Q)\) to be the subgroup of \(I(V,f)\) and \(I(V,Q)\) of determinant one maps. Given an arbitrary basis, the corresponding matrix groups of (i)-(vi) above are denoted:
\[SL_n(q),\: Sp_{2m}(q),\: SU_n(q^{1/2}),\: SO^{\circ}_{2m+1}(q),\: SO^{+}_{2m}(q),\: SO^{-}_{2m}(q)\]
We note that the elements of the symplectic group are already determinant one, hence the notation does not change.

\item For an orthogonal form \(Q\), we define \(\Omega(V,Q)\) to be the derived subgroup of \(I(V,Q)\). Given an arbitrary basis, the  corresponding matrix groups are denoted: 
\[\Omega^{\circ}_{2m+1}(q),\: \Omega^{+}_{2m}(q),\: \Omega^{-}_{2m}(q)\]

\item We have already defined the set of similarities of a form - \(\Delta(V,f)\) and \(\Delta(V,Q)\) - these too form a group. Given an arbitrary basis, the corresponding matrix groups are denoted: 
\[GL_n(q),\: GSp_{2m}(q),\: GO^{\circ}_{2m+1}(q),\: GO^{+}_{2m}(q),\: GO^{-}_{2m}(q)\]

We note that the similarity group of the zero form is the same as its isometry group, hence the notation stays the same. Also, when \(f\) is a unitary form, there is no formal notation for the matrix group corresponding to \(\Delta(V,f)\), but these groups are isomorphic to \(GU_n(q^{1/2})\circ C_{q-1}\).

\item If \(G\) is a subgroup of \(GL(V)\) and \(Z\) is the scalars contained in \(G\), then the \textit{projective group} of \(G\) is \(G/Z\). Given an arbitrary basis, the corresponding matrix notation for these groups is the matrix notation of the group \(G\) with a \(P\) placed in front, i.e. \[PGU_n(q^{1/2}),\: PGO^{\circ}_{2m+1}(q),\: PSL_n(q),\: P\Omega^{+}_{2m}(q)\:\text{ etc. }\]  
\end{enumerate}
\end{definition}

We are now ready to formally define the classical groups. The definition provided here excludes a small number of automorphisms that are often otherwise included in some classical groups (compare with the definitions in (\cite{kleidman_subgroup_1990}, p.13-14) and (\cite{bray_maximal_2013}, p.27-31)). We do so as such an exclusion makes no difference to Aschbacher's theorem for the case of \(GL(V)\).

\begin{definition}
\hypertarget{0.2.9}{A} group \(G\) is called a \textit{classical group} if it satisfies one of the following.
\begin{enumerate}[label=(\roman*),ref=(\roman*)]

\item \(S(V,f)\leq G \leq \Delta(V,f)\), where \(f\) is either the zero form or a symplectic or unitary form.

\item \(\Omega(V,Q)\leq G \leq \Delta(V,Q)\), where \(Q\) is an orthogonal form of plus, minus or \(\circ\)-type. 
\item \(G\) is the projective group of any group satisfying (i) or (ii).

\end{enumerate}

\end{definition}

This definition encompasses \(I(V,f)\), for all classical forms \(f\) (and \(Q\)). However, as the focus of our paper lies mainly with the case where \(f\) is the zero form, and our interest in the other classical groups will be majoritively in relation to the eighth Aschbacher class of \(GL(V)\), we only state a few relevant results concerning the classical groups in general. The reader is referred to \cite{cameron_notes_2000} and \cite{kleidman_subgroup_1990} to gain a fuller understanding of these groups. Proofs for the next two results can be found in (\cite{kleidman_subgroup_1990}, p.43-46).

\begin{lemma}
\hypertarget{0.2.10}{The} following isomorphisms of classical groups hold.
\begin{enumerate}[label=(\roman*),ref=(\roman*)]

\item \(PSL_2(2) \cong S_3\)
\item \(PSL_2(3) \cong A_4\)
\item \(SL_2(q)\cong Sp_2(q)\cong SU_2(q)\)
\item \(O^{\pm}_2(q)\cong D_{2(q\mp 1)}\) 
\item \(SO^{\pm}_2(q)\cong C_{q\mp 1}.C_{(2,q)}\)
\item For \(q\) odd, \(\Omega^{\circ}_3(q)\cong PSL_2(q)\)
\item \(\Omega^{+}_4(q)\cong SL_2(q)\circ SL_2(q)\)
\item \(\Omega^{-}_4(q)\cong PSL_2(q^2)\)
\item For \(q\) odd, \(\Omega^{\circ}_5(q)\cong PSp_4(q)\)
\item \(P\Omega^{+}_6(q)\cong PSL_4(q)\)
\item \(P\Omega^{-}_6(q)\cong PSU_4(q)\)
\end{enumerate}
\end{lemma}

The following result states exactly which classical groups are simple, a proof of which can found in (\cite{cameron_notes_2000}, ch.2,4-6).

\begin{lemma}
\hypertarget{0.2.11}{Let} \(V\) be a \(n\)-dimensional \(\mathbb{F}_q\)-vector space and let \(f\) and \(Q\) be classical forms. If \(G=I(V,f)\) or \(I(V,Q)\), then \(G'/Z(G)\) is simple whenever one of the following hold.

\begin{enumerate}[label=(\roman*),ref=(\roman*)]

\item \(f\) is the zero form and \(n\geq 3\) or \(n=2\) and \(q\geq 4\).
\item \(f\) is a unitary form and \(n\geq 4\) or \(n=3\) and \(q\geq 3\).
\item \(f\) is a symplectic form and \(n\geq 5\) or \(n=4\) and \(q\geq 3\).
\item \(Q\) is an orthogonal form and \(n\geq 7\).
\end{enumerate}
\end{lemma}

\begin{remark}
This lemma should be understood in view of the isomorphisms stated in \hyperlink{0.2.10}{(0.2.10)}. For example, there are unitary forms that yield a simple group when \(n=2\), but by \hyperlink{0.2.10}{(0.2.10.iii)}, these groups are accounted for in the zero form case in part (i) of \hyperlink{0.2.11}{(0.2.11)}. 
\end{remark}

\subsubsection*{Normal and characteristic subgroups}

We now state some definitions and results regarding normal and characteristic subgroups. Proofs of the first two results can be found in (\cite{gorenstein_finite_1980}, p.17-20).

\begin{definition}
Let \(X\), \(Y\) and \(Z\) be groups. We define the shorthand notation \([X,Y,Z]\) to be the commutator \([[X,Y],Z]\).
\end{definition}

\begin{lemma}[Three Subgroup Lemma]
Let \(G\) be a group. \hypertarget{0.2.14}{If} \(X, Y, Z\leq G\) and \(L\unlhd G\) such that \([X,Y,Z]\leq L\) and \([Y,Z,X]\leq L\), then \([Z,X,Y]\leq L\).
\end{lemma}

\begin{prop}
\hypertarget{0.2.15}{If} \(L\) is a non-abelian minimal normal subgroup of \(G\), then \(L=Y_1\times \cdots \times Y_k\), where the \(Y_j\) are non-abelian simple subgroups of \(L\) that are conjugate in \(G\).   
\end{prop}

Recall that for all groups \(H\) and \(G\) such that \(H\leq G\), there is a canonical map \(N_G(H)\longrightarrow Aut(H)\) with kernel \(C_G(H)\). Accordingly, we make use of the following shorthand notation.

\begin{definition}
\hypertarget{0.2.16}{Let} \(H\leq G\). We define \(Aut_G(H)\) to be the quotient group \(N_G(H)/C_G(H)\).
\end{definition}

The next result relates to the normaliser of particular subgroups of \(GL(V)\), which we will make use of in section eight. See (\cite{hup}, p.187-189) for a proof.

\begin{prop}
\hypertarget{0.2.17}{Let} \(F\) be a field of order \(q\), let \(V\) be a \(n\)-dimensional \(F\)-vectorspace and let \(H\leq GL(V,F)\) be a cyclic subgroup of order \(q^n-1\). If \(h\in H\) such that the order of \(h\) does not divide \(q^m-1\), for all \(m\) dividing (but not equal to) \(n\), then \(N_{GL(V)}(\langle h\rangle)=N_{GL(V)}(H)\).
\end{prop}

Next, we examine a particular subgroup that is of great importance for the proof of the main theorem. See (\cite{aschbacher_finite_2000}, p.156-159) for proofs of the stated results.

\begin{definition}
Let \(G\) be a group. A subgroup \(H\leq G\) is called a \textit{subnormal subgroup}, denoted \(H\unlhd\unlhd\, G\), if there is a chain: 
\[H=G_0\unlhd G_1\unlhd \cdots \unlhd G_k=G\]
\end{definition}

\begin{definition}
\hypertarget{quasi}{A} group \(G\) is called \textit{quasisimple} if it is perfect and \(G/Z(G)\) is simple.
\end{definition}

\begin{definition}
A \textit{component} of a group \(G\) is a quasisimple subnormal subgroup.
\end{definition}

Components of a group have the following two properties.

\begin{prop}
\hypertarget{0.2.21}{Let} \(H\unlhd\unlhd\, G\). The components of \(H\) are the components of \(G\) that are also contained in \(H\).
\end{prop}

\begin{prop}
\hypertarget{0.2.22}{If} \(H\) is a component of a group \(G\), then:
\begin{enumerate}[label=(\roman*),ref=(\roman*)]

\item \(H\) commutes with all other components of \(G\).
\item \(H\) commutes with all \(H\)-invariant solvable subgroups of \(G\).
\end{enumerate}

\end{prop}

These definitions allow us to define the following subgroup.

\begin{definition}
Let \(H_1,...\,,H_k\) be the components of a group \(G\). We define the \textit{layer} of \(G\) to be the group \(E(G)=H_1H_2\cdots H_k = H_1\circ...\,\circ H_k\).
\end{definition}

\begin{remark}
\hypertarget{0.2.24}{By} \hyperlink{0.2.11}{(0.2.11.i)}, whenever \(n\geq 3\), \(SL_n(q)\) is quasisimple and \(E(GL_n(q))=SL_n(q)\). 
\end{remark}

\begin{prop}
\hypertarget{0.2.25}{The} layer is a characteristic subgroup.
\end{prop}

Part (i) in the statement of the main theorem says that our subgroup \(H\) has a unique normal quasisimple subgroup. In the language of our previous discussion, this is equivalent to saying that \(E(H)\) has a unique component.

\subsection{Modules and Representation Theory}

Here we discuss \(FG\)-modules and \(F\)-representations of a group \(G\); looking at some of their properties and establishing the link between them. In particular, we work towards an understanding of \textit{absolute irreducibility} of a module or representation.

\subsubsection*{FG-Modules}

We begin by recalling the definition of a \(FG\)-module.

\begin{definition}
\hypertarget{0.3.1}{Let} \(G\) be a group, \(F\) a field and \(V\) a \(F\)-vector space. We call \(V\) a \textit{\(FG\)-module} if there exists a map \(G\times V\longrightarrow V;\,(g,v)\mapsto g\cdot v\) (referred to as a \textit{\(G\)-action}) such that:
 \begin{enumerate}[label=(\roman*),ref=(\roman*)]
 \item \(g\cdot (v+w)=g\cdot v + g\cdot w\).
 \item \((g+h)\cdot v = g\cdot v + h \cdot v\).
 \item \(g\cdot (h\cdot v)=gh\cdot v\).
 \item \(g\cdot \lambda v= \lambda (g\cdot v)\).
 \item \(1_G\cdot v=v\).
 \end{enumerate}
for all \(g,h\in G\), \(v,w\in V\) and \(\lambda\in F\).
\end{definition}

\begin{definition}
Let \(V\) and \(W\) be \(FG\)-modules. We call a map  \(\phi:V\longrightarrow W\) a \textit{\(FG\)-homomorphism} if it is \(F\)-linear on the underlying vector spaces \(V\) and \(W\) and \(\phi(g\cdot v)=g\cdot\phi(v)\), for all \(g\in G\) and \(v\in V\). 
If in addition, \(\phi\) is an isomorphism on the underlying vector spaces, we call it an \textit{\(FG\)-isomorphism}.  
\end{definition}

\begin{definition}
Let \(V\) and \(W\) be \(FG\)-modules. We denote the set of all \(FG\)-homomorphisms from \(V\) to \(W\) by \(Hom_{FG}(V,W)\). If \(V=W\), then \(Hom_{FG}(V,W)\) is a ring, usually denoted \(End_{FG}(V)\).  
\end{definition}

\begin{definition}
Let \(V\) and \(W\) be \(FG\)-modules. We define a \textit{\(FG\)-module tensor product} to be the tensor product of the underlying vector spaces equipped with the \(G\)-action defined on a basis vector by \(g\cdot(v_i\otimes w_j)=g\cdot v_i\otimes g\cdot w_j\) and extended to the whole space by \(F\)-linearity. 
\end{definition}

The next two propositions show how \(Hom_{FG}\) and \(FG\)-module tensor products interact with direct sums. See (\cite{rotman_introduction_1979}, p.29-34) for proofs. 

\begin{prop}
\hypertarget{0.3.5}{If} \(V\), \(W\) and \(U\) are \(FG\)-modules, then:
\begin{enumerate}[label=(\roman*),ref=(\roman*)]
\item \(Hom_{FG}(V\oplus W, U)\cong Hom_{FG}(V, U)\oplus Hom_{FG}(W, U)\).
\item\(Hom_{FG}(V, W\oplus U)\cong Hom_{FG}(V, W)\oplus Hom_{FG}(V, U)\).
\end{enumerate}

\end{prop}

\begin{prop}
\hypertarget{0.3.6}{If} \(V\), \(W\) and \(U\) are \(FG\)-modules, then:
\[V\otimes (W\oplus U)\cong(V\otimes W)\oplus (V\otimes U)\]
\end{prop}

If \(G\) is a direct product, then it can also give rise to a different \(FG\)-module tensor product, which we define below.

\begin{definition}
\hypertarget{0.3.7}{Let} \(G_1,...\,,G_m\) be groups and let \(V_i\) be a \(FG_i\)-module, for all \(1\leq i \leq m\). We define an action of \(G=G_1\times \cdots \times G_m\) on the vector space \(V_1\otimes \cdots \otimes V_m\) by \((g_1,...\,,g_m)\cdot (v_1\otimes \cdots \otimes v_m)=g_1\cdot v_1\otimes \cdots \otimes g_m\cdot v_m\), this action makes \(V_1\otimes \cdots \otimes V_m\) a \(FG\)-module.
\end{definition}

The next result allows us to relate a \(FG\)-module of a direct product to that of a central product, a proof is found in (\cite{gorenstein_finite_1980}, p.102).

\begin{prop}
\hypertarget{0.3.8}{If} \(G\) is a group and \(N\) is a normal subgroup, then \(F[G/N]\)-modules are in one-to-one correspondence with \(FG\)-modules on which \(N\) acts trivially. 
\end{prop}

If \(K\) is a finite field extension of \(F\), then \(K\) can be viewed as a \(F\)-vector space and if \(V\) is a \(n\)-dimensional \(F\)-vector space, we may then define the vector space tensor product \(V\otimes_F K\). Moreover, this can be viewed as a \(n\)-dimensional vector space over \(K\) as follows. If \(\{v_1,...\,,v_n\}\) is a \(F\)-basis for \(V\) and \(K\)-multiplication on \(V\otimes_F K\) is defined by \(k'\cdot (v\otimes k)=v\otimes k'k\), then \(V\otimes_F K\) is spanned by \(\{v_i\otimes 1|\, 1\leq i \leq n\}\) over \(K\) and these vectors are clearly \(K\)-linearly independent. This construction can be thought of as a way of extending the scalars of \(V\) to incorporate \(K\).

\begin{definition}
\hypertarget{0.3.9}{Let} \(K\) be a finite field extension of \(F\) and let \(V\) be a \(FG\)-module. We define \(V^K\) to be the \(KG\)-module \(V\otimes_F K\) on which \(g\in G\) acts by \(g\cdot (v\otimes k)=(g\cdot v)\otimes k\). 
\end{definition}

If \(V\) is an irreducible \(FG\)-module, it does not immediately follow that \(V^K\) is an irreducible \(KG\)-module. This motivates the following definition.

\begin{definition}
\hypertarget{0.3.10}{An} irreducible \(FG\)-module \(V\) is said to be \textit{absolutely irreducible} if \(V^K\) is an irreducible \(KG\)-module for every field \(K\) containing \(F\). 
 \end{definition}
 
The next lemma is a key result for proving absolute irreducibility and it will be used many times throughout our paper. See (\cite{curtis_representation_1988}, p.202-203) for a proof. 

\begin{lemma}
\hypertarget{0.3.11}{If} \(V\) is an irreducible \(FG\)-module, then the following are equivalent.
\begin{enumerate}[label=(\roman*),ref=(\roman*)]
\item \(V\) is absolutely irreducible.
\item \(End_{FG}(V)=F\).
\item \(C_{GL(V)}(G)=F^{\times}\).
\end{enumerate}
\end{lemma}

The following result states the absolute irreducibility of some of the classical groups that concern us most. See (\cite{kleidman_subgroup_1990}, p.50-51) for a proof.

\begin{prop}

\hypertarget{0.3.12}{If} \(V\) is a \(F\)-vector space with a classical form \(f\) or \(Q\), then:
\begin{enumerate}[label=(\roman*),ref=(\roman*)]
\item \(S(V,f)\) or \(S(V,Q)\) acts absolutely irreducibly on \(V\) if and only if it is not isomorphic to \(SO^{\pm}_{2}(q)\), where \(q\) is odd.
\item\(I(V,f)\) or \(I(V,Q)\) acts absolutely irreducibly on \(V\) if and only if it is not isomorphic to \(O^{+}_{2}(2)\) or \(O^{+}_{2}(3)\).
\end{enumerate}

\end{prop}

\subsubsection*{Representation Theory}

Let \(G\) be a group and \(F\) a field. We will use the terminology \textit{\(F\)-representation} to refer to a homomorphism \(\rho:G\longrightarrow GL(V,F)\) and the terminology \textit{matrix representation} to refer to a homomorphism \(\rho:G\longrightarrow GL_n(q)\). There is an obvious correspondence between the two and when working with representations we will often switch between them, making use of whichever eases notation in a given context. There is also a correspondence between \(FG\)-modules and \(F\)-representations of \(G\). If \(V\) is a \(FG\)-module, then every element \(g\in G\) induces a linear map \(\varphi_g:V\longrightarrow V\) with inverse \(\varphi_{g^{-1}}\), thus \(\rho: G\longrightarrow GL(V,F);\,g\mapsto \varphi_g\) is a \(F\)-representation. On the other hand, if \(\rho: G\longrightarrow GL(V,F)\) is a \(F\)-representation, then the map \(G\times V\longrightarrow V:\,(g,v)\mapsto \rho(g)(v)\) is a \(G\)-action satisfying \hyperlink{0.3.1}{(0.3.1)}, thus \(V\) is a \(FG\)-module.

The reader should be familiar with the basic results of representation over \(\mathbb{C}\). However, our interest is in representations over a finite field \(F\) of characteristic \(p\) and it is not true in general that the same results hold over such a field. There is however, a ‘nice’ case where many of the fundamental result of \(\mathbb{C}\)-representation theory hold; this is when \(p\nmid |G|\). As this happens to be the only case we will need to apply such results to, we will not delve into the background theory here, but the reader is referred to (\cite{isaacs_character_2006}, p.262-269) for a gentle introduction to the theory of representations over fields of prime characteristic and how they relate to \(\mathbb{C}\)-representations. The reader is also referred to the same reference for a concrete justification that when \(F\) is a field of characteristic \(p\) and \(G\) is a group such that \(p\nmid |G|\), the following three results hold.

\begin{prop}
\hypertarget{0.3.13}{A} group \(G\) has the same number of irreducible \(F\)-representations as it has conjugacy classes.
\end{prop}
 
\begin{prop}
\hypertarget{0.3.14}{A} group \(G\) has \(|G/G'|\) \(1\)-dimensional irreducible \(F\)-representations.
\end{prop}

\begin{prop}
\hypertarget{0.3.15}{If} \(G\) is a group and \(\{\rho_i|1\leq i \leq k\}\) is a set of representatives of the irreducible \(F\)-representations of \(G\), then:
\[|G|=\sum\limits_{i=1}^{k} deg(\rho_i)^2\]
\end{prop}

Recall Schur's Lemma. See (\cite{serre_linear_1996}, p.13) for a proof.

\begin{lemma}[Schur's Lemma]
Let \(G\) be a group, let \(V\) and \(W\) be \(F\)-vector spaces and let \(\rho_V:G\longrightarrow GL(V)\) and \(\rho_W:G\longrightarrow GL(W)\) be irreducible representations.
\begin{enumerate}[label=(\roman*),ref=(\roman*)]
\item If \(V\ncong W\), then \(Hom_{FG}(V,W)=\{0\}\).  
\item If \(V\cong W\) and \(F\) is algebraically closed, then \(Hom_{FG}(V,W)=\{\lambda\cdot id\,|\,\lambda\in F\}\).
\end{enumerate}

\end{lemma}

Next we will explore when the image of a representation fixes a particular classical form on \(V\). In order to begin this discussion we need two definitions, the first of which is the analogue of absolute irreducibility for a \(F\)-representation. 

\begin{definition}
An irreducible representation \(\rho: G\longrightarrow GL(V,F)\)
is said to be \textit{absolutely irreducible} if the representation \(\rho_K: G\longrightarrow GL(V^K,K)\), where \(\rho_K(g)(v\otimes k)=\rho(g)(v)\otimes k\), is irreducible for every field \(K\) containing \(F\).
\end{definition}

\begin{definition}
Let \(\rho:G\longrightarrow GL_n(q)\) be a matrix representation. If \(\theta\in Aut(\mathbb{F}_q)\), then \(\rho^{\theta}:G\longrightarrow GL_n(q)\) is the representation defined by \(\rho^{\theta}(g)=(\rho(g))^{\theta}\).
\end{definition}

A proof of the next three results can be found in (\cite{kleidman_subgroup_1990}, p.48-56).

\begin{prop}
\hypertarget{0.3.19}{If} \(\rho:G\longrightarrow GL_n(p^e)\) is an absolutely irreducible representation, then:
\begin{enumerate}[label=(\roman*),ref=(\roman*)]
\item  \(\rho(G)\) fixes a unitary form if and only if \(e\) is even and \(\rho^{\theta}\) is equivalent to the dual representation \(\rho^{\ast}\), where \(\theta\) is the field automorphism of order two.
\item \(\rho(G)\) fixes a symplectic or non-degenerate symmetric bilinear form if and only if \(\rho\) is equivalent to the dual representation \(\rho^{\ast}\).

\end{enumerate}
\end{prop}

\begin{prop}
Let \(\rho:G\longrightarrow GL_n(q)\) be an absolutely irreducible representation. If \(\rho(G)\) fixes two symplectic, unitary or non-degenerate symmetric bilinear forms, then they are equal up to scalar multiplication.

\end{prop}

\begin{prop}
\hypertarget{0.3.21}{Let} \(\rho:G\longrightarrow GL(V)\) be an absolutely irreducible representation. If the image of \(\rho\) fixes a symplectic, unitary or non-degenerate symmetric bilinear form \(f\), then \(N_{GL(V)}(\rho(G))\leq \Delta(V,f)\).  If in addition \(\rho(G)=I(V,f)\), then equality holds. 

\end{prop}

\section{Aschbacher Class \(\mathcal{C}_1\) - Subspace Stabilisers }

In the following eight sections, we will discuss each of the eight Aschbacher classes. In each section, we will provide a formal definition for the members of that class, explain their group structure, discuss the differences between the definitions stated here and those found in Aschbacher's paper as well as stating and proving any results necessary for the corresponding Lemma in the proof of the main theorem. 

In this first section, we will tackle the first class; beginning with the following definition. 

\begin{defn}
Let \(W\subset V\) be a proper non-trivial subspace. We will define \(N_{GL(V)}(W)\) to be the group of all \(g\in GL(V)\) such that \(g(W)= W\). 
\end{defn}

If \(\{v_1,...\,v_k\}\) is a basis for a subspace \(W\), we can extend this set by some elements \(v_{k+1},...\,, v_n \in V\) to form a basis for \(V\). With respect to this basis, an element of \(GL(V)\) that stabilises \(W\) takes the form of a block matrix:
\[\begin{pmatrix} A & B \\ 0 & D\end{pmatrix}\]

\noindent where \(A\in GL_{k}(q)\), \(D\in GL_{m}(q)\)  and \(B\in M_{k,m}(q)\), with \(m=n-k\). The subgroup \(G\leq GL_n(q)\) of all such matrices is therefore isomorphic to \(N_{GL(V)}(W)\). We can identify two important subgroups of this group \(G\).
\[Q:=\{\begin{pmatrix} I_k & B \\ 0 & I_m \end{pmatrix}| B\in M_{k,m}(q)\} \text{ }\text{ }\text{ }\text{    and    }\text{ }\text{ }\text{ } L:=\{\begin{pmatrix} A & 0 \\ 0 & D \end{pmatrix}| A\in GL_{k}(q), D\in GL_{m}(q)\}\]
These are referred to as the \textit{unipotent radical} and \textit{Levi complement} respectively. The following properties are easily observed.

\begin{pro}
If \(L,Q,G\leq GL_n(q)\) are defined as above, then:
\begin{enumerate}[label=(\roman*),ref=(\roman*)]

\item \(Q\cong \mathbb{F}_q^{km}\) and \(L\cong GL_k(q)\times GL_m(q)\) 
\item \(Q\cap L = 1\) 
\item \(Q\unlhd G\)
\item \(QL\cong G\) 
\end{enumerate}
\end{pro}

\begin{proof}
Parts (i) and (ii) are clear from the definition. Part (iii) follows from block matrix multiplication. Part (iv) follows from (ii) and (iii). 
\end{proof}

We are now ready to define the first Aschbacher class.

\begin{defn}
A group \(G\leq GL(V)\) is a member of \(\mathcal{C}_1\) if \(G=N_{GL(V)}(W)\) for some proper non-trivial subspace \(W\subset V\). Such groups are isomorphic to \(\mathbb{F}_q^{km}\rtimes (GL_k(q)\times GL_m(q))\), where \(k\) is the dimension of \(W\) and \(m=n-k\).
\end{defn}

\begin{rem}
\hypertarget{1.4}{This} first class is significantly more complex in Aschbacher's original paper, owing to two reasons.
\begin{enumerate}[label=(\roman*),ref=(\roman*)]

\item Aschbacher's main theorem applies to groups related to each of the classical forms. When dealing with such a variety of forms, there is a need to distinguish between subspaces of \(V\) on which the form acts as the zero form (the formal language is \textit{totally singular}) and subspaces of \(V\) on which the form acts as a non-degenerate form. The extra conditions in Aschbacher's \(\mathcal{C}_1\) class are all to account for subspaces of \(V\) on which the specified form acts as a non-degenerate form. In our case, the only form we are concerned with is zero on the whole of \(V\), thus we need not be concerned with these additional conditions.

\item Aschbacher additionally defines a supplementary class \(\mathcal{C}'_1\) for dealing with a particular case when \(n>2\) and the classical group in question is a subgroup of \(Aut(SL_n(q))\) that contains the inverse-transpose automorphism (one of those excluded in \hyperlink{0.2.9}{(0.2.9)}). 
However, we are only concerned with the group \(GL_n(q)\), which does not contain this automorphism when \(n>2\), and hence we need not encompass the class \(\mathcal{C}'_1\) into our definition.     

\end{enumerate}

\end{rem}

\section{Aschbacher Class \(\mathcal{C}_2\) - Decomposition Stabilisers}

If \(V=\bigoplus_{i=1}^k V_i\) is a direct sum decomposition (we will be assuming \(k>1\)) in which each summand has dimension \(m\), we will refer to it as a \textit{\(m\)-decomposition}. 

\begin{defn}
Let \(V=\bigoplus_{i=1}^k V_i\) be a \(m\)-decomposition. A group \(G\leq GL(V)\) is said to \textit{stabilise} this decomposition if \(G\) permutes the summands \(V_1, ...\, ,V_k\). If \(G\) is the maximal group with this property, we call it a \textit{\(m\)-decomposition stabiliser}, denoted by \(N_{GL(V)}(\{V_1,...\, ,V_k\})\).
\end{defn}

The following result shows the structure of a \(m\)-decomposition stabiliser.

\begin{pro}
If \(V=\bigoplus_{i=1}^k V_i\) is a \(m\)-decomposition, then \(N_{GL(V)}(\{V_1,...\, ,V_k\})\cong (GL(V_1)\times \cdots \times GL(V_k))\rtimes S_k\cong GL(V_1)\wr S_k\).
\end{pro}

\begin{proof}
Let \(N=N_{GL(V)}(\{V_1,...\, ,V_k\})\) and \(G=GL(V_1)\times \cdots \times GL(V_k)\). Define the homomorphism \(\phi_1:G\longrightarrow N\) such that \(\phi_1((g_1,...\,,g_k))(v_1,...\,,v_k)=(g_1(v_1),...\,,g_k(v_k))\), this map is evidently faithful, thus we have an embedding \(G\xhookrightarrow{} N\). Next we define the homomorphism \(\phi_2:S_k\longrightarrow N\) such that \(\phi_2(\sigma)(v_1,...\,,v_k))=(v_{\sigma^{-1}(1)},...\,,v_{\sigma^{-1}(k)})\), since this map is also faithful, we have an embedding \(S_k\xhookrightarrow{} N\). We observe that \(\phi_1(G)\,\cap\, \phi_2(S_k)=1\). Furthermore, \(\phi_2(S_k)\) acts on \(\phi_1(G)\) by permuting its coordinates. Indeed, if \(\sigma\in S_k\), \((g_1,...\,,g_k)\in G\), \((v_1,...\,,v_k)\in \bigoplus_{i=1}^k V_i\) and for all \(1\leq i\leq k\) we define \(w_i:=g_i(v_i)\), then:
\[\phi_2(\sigma) (\phi_1((g_1, ... , g_k)) (v_1,...\,,v_k))= \phi_2(\sigma)(w_1,...\,,w_k)=(w_{\sigma^{-1}(1)},...\,,w_{\sigma^{-1}(k)})\]
\[=(g_{\sigma^{-1}(1)}(v_{\sigma^{-1}(1)}),...\,,g_{\sigma^{-1}(k)}(v_{\sigma^{-1}(k)}))=\phi_1((g_{\sigma^{-1}(1)}, ... , g_{\sigma^{-1}(k)}))\phi_2(\sigma)((v_1,...\,,v_k))\]

Thus \(\phi_2(\sigma)\phi_1((g_1, ... , g_k))\phi_2(\sigma)^{-1}= \phi_1((g_{\sigma^{-1}(1)}, ... , g_{\sigma^{-1}(k)}))\) as claimed. We can deduce then, that the semi-direct product \(\phi_1(G)\rtimes\phi_2(S_k)\) with respect to this action, is a subgroup of \(N\). 

To show the reverse containment, let \(h\in N\). There exists \(\sigma\in S_k\) such that \(h(V_i)=V_{\sigma^{-1}(i)}\), for all \(1\leq i\leq k\). Thus \(h\phi_2(\sigma)\) is an element in \(\phi_1(G)\) and the result follows.  
\end{proof}

\begin{defn}
A group \(G\leq GL(V)\) is a member of \(\mathcal{C}_2\) if there exists a \(m\)-decomposition \(V=\bigoplus_{i=1}^k V_i\), such that \(G=N_{GL(V)}(\{V_1, ...\, ,V_k\})\). Such groups are isomorphic to \(GL_m(q)\wr S_k\).   
\end{defn}

In the proof of the main theorem, \hyperlink{lemma 2}{Lemma 2} relies on a well known-result by Alfred H. Clifford. Before stating this, we must define some further concepts in representation theory.
 
 \begin{defn}
Let \(F\) be a field, let \(G\) be a group and let \(V\) be a \(FG\)-module with \(\bigoplus_{i=1}^k V_i\) a decomposition of \(V\) into its irreducible \(FG\)-modules. For a fixed integer \(1\leq j\leq k\), we define the \textit{homogeneous component} associated with \(V_j\) to be the direct sum of the irreducible \(FG\)-modules of \(V\) that are isomorphic to \(V_j\).
 \end{defn}
 
 \begin{defn}
 Let \(N\) be a normal subgroup of a group \(G\). Two matrix representations \(\rho\), \(\rho'\)\(:N\xrightarrow{} GL_n(q)\) are said to be \textit{conjugate} in \(G\), if there exist \(g\in G\) such that for all \(n\in N\), we have \(\rho'(n)=\rho(g^{-1}ng)\). Two \(FN\)-modules are said to be \textit{conjugate} in \(G\) if there exists a basis for each, such that the corresponding matrix representations are conjugate.
 \end{defn}
 
We now state Clifford's theorem, a proof of which can be found in (\cite{gorenstein_finite_1980}, p.70-72).
 
 \begin{thm}[Clifford's Theorem]
 \hypertarget{2.6}{Let} \(V\) be an irreducible \(FG\)-module and \(N\unlhd G\). If \(V_1\subseteq V\) is an irreducible \(FN\)-submodule, then:
 \begin{enumerate}[label=(\roman*),ref=(\roman*)]
 
 \item\(V=\bigoplus_{i=1}^k V_i\), where the \(V_i\) are irreducible \(FN\)-modules each conjugate to \(V_1\) in \(G\).
 
 \item \(G\) permutes the \(FN\)-homogeneous components of \(V\) transitively.
 
 \item Each \(FN\)-homogeneous component is stabilised by \(C_G(N)N\). 
 \end{enumerate}
 \end{thm}

\section{Aschbacher Class \(\mathcal{C}_3\) - Extension Field Stabilisers}
 
Let \(K\) be a finite extension field of \(F=\mathbb{F}_q\) such that \(|K:F|=r\). We can view \(K\) as a \(r\)-dimensional \(F\)-vector space and then if \(V\) is a \(m\)-dimensional \(K\)-vector space, \(V\) can be viewed as a \(mr\)-dimensional \(F\)-vector space. Since the maps in \(GL(V,K)\) are \(K\)-linear, they must also be \(F\)-linear, and so \(GL(V,K)\) is a subgroup of \(GL(V,F)\). Since \(Z(GL(V,K))\cong K^{\times}\), we make use of the convention of identifying \(Z(GL(V,K))\leq GL(V,F)\) by \(K^{\times}\). Accordingly, this process can be thought of as embedding the extension field \(K\) into \(GL(V,F)\). We also note that \(C_{GL(V,F)}(K^{\times})=GL(V,K)\).

Considering \(V\) as a \(K\)-vector space, the field automorphisms of \(K\) sending \(k\mapsto k^{q^j}\), for \(1\leq j\leq r-1\), induce maps \(\phi_{q^j}:V\longrightarrow V\) defined by:
\[\phi_{q^j}(\sum_{i=1}^{m}\lambda_iv_i)=\sum_{i=1}^{m}\lambda_i^{q^j}v_i\]
\noindent where \(\lambda_i\in K\) and \(\{v_1,...\,,v_m\}\) is a \(K\)-basis of \(V\). Furthermore, for all \(k\in K^{\times}\), we have:
\[\phi_{q^j}^{-1}k\phi_{q^j}(\sum_{i=1}^{m}\lambda_iv_i)=\phi_{q^j}^{-1}(\sum_{i=1}^{m}k\lambda_i^{q^j}v_i)=k^{q^{r-j}}(\sum_{i=1}^{m}\lambda_iv_i)\]
Since \(q\) is the order of \(F\), these maps are also \(F\)-linear and so \(\langle \phi_q\rangle\) is a subgroup of \(GL(V,F)\) that normalises \(K^{\times}\). Therefore \(GL(V,K)\langle \phi_q\rangle\leq N_{GL(V,F)}(K^{\times})\).

We can show, on the other hand, that this is the whole normaliser of \(K^{\times}\). Let \(g\in GL(V,F)\) be a element that normalises, but does not centralise, \(K^{\times}\). If \(k,k'\in  K^{\times}\), then \(g^{-1}(k+k')g=g^{-1}kg+g^{-1}k'g\) and hence \(g\) preserves, not just the multiplicative structure of \(K^{\times}\), but also the additive structure. Thus, all such elements embed into \(Aut(K)\) (where \(K\) is the field here). Since these elements are \(F\)-linear, their images fix the subfield \(F\), thus the image of all such maps is a subgroup of \(\{\alpha_j: k\mapsto k^{q^j}\,|\,0\leq j \leq r-1\}\cong C_r\). Therefore \(|N_{GL(V)}(K^{\times})|\leq |C_{GL(V,F)}(K^{\times})||C_r|\) and we may conclude that \(N_{GL(V,F)}(K^{\times})=GL(V,K)\langle \phi_q\rangle\).

We are now ready to state the main definition of this section. 

\begin{defn}
\hypertarget{3.1}{A} group \(G\) is a member of \(\mathcal{C}_3\) if \(G=N_{GL(V,F)}(K^{\times})\), where \(K\) is a finite field extension of \(F\) such that \(|K:F|=r\) is a prime divisor of \(n\). Such groups are isomorphic to \(GL_m(q^r)C_r\), where \(m=\tfrac{n}{r}\). 
\end{defn}
 
 \begin{rem} We make the following remarks about the definition above.
 \begin{enumerate}[label=(\roman*),ref=(\roman*)]
 
 \item If \(|K:F|=p_1p_2\cdots p_t\), where each \(p_i\) is prime, then \(K\) has a unique subfield \(K_i\) such that \(|K_i:F|=p_i\) and therefore \(K_i^{\times}\) is characteristic in \(K^{\times}\). So any element in \(GL(V,F)\) normalising \(K^{\times}\) also normalises \(K_i^{\times}\) and therefore \(N_{GL(V,F)}(K^{\times})\) is not maximal in general. This explains why we require the extension field \(K\) to be of prime index.
 
 \item Aschbacher includes an additional condition for \(G\) to be a member of \(\mathcal{C}_3\); namely that \(C_{I(V,f)}(K^{\times})\) acts irreducibly on \(V\), where \(f\) (or \(Q\)) is the specified classical form. However, in our case \(I(V,f)=GL(V,F)\) and \(C_{GL(V,F)}(K^{\times})\cong GL(V,K)\) acts transitively on the non-zero vectors of \(V\). So for us, \(C_{I(V,f)}(K^{\times})\) is always irreducible on \(V\). 
 \end{enumerate}
 \end{rem}

For the rest of this section, let \(G\) be a group and \(V=\bigoplus^m_{i=1} V_i\) be a homogeneous \(FG\)-module, such that the \(V_i\) are isomorphic irreducible submodules of dimension \(d\). We define \(E:=End_{FG}(V_1)\) and identify \(F\) with the subring of scalar maps on \(V_1\). We will show that \(E\) is a finite field extension of \(F\) that can be embedded into \(GL(V,F)\). We begin by stating a well-known result by Joseph Wedderburn, for a proof of which we refer the reader to (\cite{lidl_finite_1983}, p.70-71).

\begin{lem}[Wedderburn's Little Theorem]
\hypertarget{wed}{A} finite division ring is a field.
\end{lem}

\begin{pro}
\(E\) \hypertarget{3.4}{is} field.
\end{pro}
   
\begin{proof}
We know that \(E\) is a finite ring. By Schur's Lemma, the elements of \(E\) are isomorphisms, so it is a division ring and therefore by Wedderburn's little theorem, \(E\) is a field. 
\end{proof}

\begin{pro}
\(C_{GL(V)}(G)\cong GL_m(E)\), \hypertarget{3.5}{where} \(m=\tfrac{n} {d}\).
 \end{pro}
  
 \begin{proof}
 Define \(W:=Hom_{FG}(V_1,V)\). Since \(V=\bigoplus^m_{i=1} V_i\cong V_1^{\oplus\, m}\) and in view of \hyperlink{0.3.5}{(0.3.5)}, we observe the \(FG\)-isomorphisms:
 \[W=Hom_{FG}(V_1,\bigoplus^m_{i=1} V_i)\cong \bigoplus^m_{i=1}Hom_{FG}(V_1,V_i)\cong E^m\]
 Hence, \(W\) is a \(m\)-dimensional \(E\)-vector space, with scalar multiplication of \(E\) defined by right composition of maps i.e \(e\cdot w=w\circ e\), for \(e\in E\) and \(w\in W\). We will construct an isomorphism \(C_{GL(V)}(G)\xrightarrow{\text{ }\sim\text{ }}GL(W,E)\cong GL_m(E)\).
 
 For all \(1\leq i\leq m\), we fix \(FG\)-isomorphisms \(\alpha_i:V_1\longrightarrow V_i\) and then \(\mathcal{B}=\{\alpha_1,...,\alpha_m\}\subset W\) is an \(E\)-linearly independent subset of size \(m\) and hence it is an \(E\)-basis of \(W\). We define an action of \(C_{GL(V)}(G)\) on \(W\) by left composition of maps i.e \(c\ast w=c\circ w\), for \(w\in W\) and \(c\in C_{GL(V)}(G)\). This action is \(E\)-linear since:
 \[c\ast (e\cdot \alpha_i)=c\ast (\alpha_i\circ e)=c\circ \alpha_i\circ e=e\cdot (c\circ \alpha_i)=e\cdot (c\ast \alpha_i)\] 
 
 Furthermore, it is faithful. Indeed, if \(c\ast \alpha_i=\alpha_i\) for all \(i\), then \(c\) fixes every vector in \(V=\bigoplus^m_{i=1} V_i\) and therefore \(c=id\). Hence, we have an embedding \(\phi: C_{GL(V)}(G)\xhookrightarrow{}GL(W,E)\). We claim that \(\phi\) is also surjective. 
 
 Let \(B=(b_{ij})\) be an arbitrary element in \(GL(W,E)\), then \(B\,\alpha_i=\Sigma_{j=1}^m b_{ij}\cdot \alpha_j\), where \(b_{ij}\in E\). Fix an \(E\)-basis \(\{v_1,\:...\:,v_t\}\) of \(V_1\), then \(\{\alpha_i(v_j)\:|\: 1\leq i \leq m,\:1\leq j \leq t\}\) is an \(E\)-basis of \(V\). Therefore \(B\) gives rise to a linear map \(g_B:V\xrightarrow{}V;\,\alpha_i(v_j) \mapsto B\,\alpha_i(v_j)\). This map is invertible, since \(B\) is invertible and therefore \(g_B\in GL(V,E)\). Let \(g\) be an arbitrary element in \(G\) with \(g( \alpha_i(v_j))=\sum_{k=1}^{t} \lambda_{ik}\cdot \alpha_i(v_k)\), where \(\lambda_{ik}\in E\), then:
 \begin{align*}
 (g\circ g_B )(\alpha_i(v_j))&=g(B\,\alpha_i(v_j))\\&=B(\,g(\alpha_i(v_j)))\\&=B(\,\Sigma_{k=1}^t \lambda_{ik}\cdot \alpha_i(v_k)) \\ &=\Sigma_{k=1}^t \lambda_{ik}( B\,\alpha_i(v_k))\\&=\Sigma_{k=1}^t \lambda_{ik}(g_B(\alpha_i(v_k))) \\ &=g_B(\Sigma_{k=1}^t \lambda_{ik}\cdot \alpha_i(v_k))=(g_B\circ g )(\alpha_i(v_j)) 
 \end{align*}
 
 Therefore \(g_B\in C_{GL(V)}(G)\) and \(\phi (g_B)=B\). Thus \(\phi\) is surjective as claimed and \(C_{GL(V)}(G)\cong GL(W,E)\cong GL_m(E)\).
 \end{proof}
 
\noindent\hypertarget{3.6}{Therefore}, we may identify \(E^{\times}\) with \(Z(C_{GL(V)}(G))\) and the next corollary follows.
 
 \begin{cor}
\(N_{GL(V)}(G)\leq N_{GL(V)}(E^{\times})\).
 \end{cor}
 
 \begin{proof}
All \(g\in N_{GL(V)}(G)\) must normalise \(C_{GL(V)}(G)\) and so it follows that they also normalise \(Z(C_{GL(V)}(G))=E^{\times}\) too.
 \end{proof}

\section{Aschbacher Class \(\mathcal{C}_4\) - Tensor Product Stabilisers}

Let \(V_1\) and \(V_2\) be \(F\)-vector spaces of dimensions \(n_1\) and \(n_2\) respectively. As defined in \hyperlink{0.3.7}{(0.3.7)}, \(GL(V_1)\times GL(V_2)\) acts naturally on the \(n_1n_2\)-dimensional vector space \(W=V_1\otimes V_2\). In general, this action will not be faithful. Indeed, if \(\lambda\) is a scalar map, then: \[(\lambda,\lambda^{-1})\cdot (v_1\otimes v_2) = \lambda v_1\otimes  \lambda^{-1}v_2= (\lambda\lambda^{-1})v_1\otimes v_2=v_1\otimes v_2\]
However, quotienting out by the central subgroup \(\{(\lambda,\lambda^{-1})\:|\:\lambda \in F^{\times})\}\) yields a central product that acts faithfully on \(W\). We denote this central product \(GL(V_1)\otimes GL(V_2)\) and refer to it as a \textit{tensor product stabiliser}. Thus, if \(n_1n_2=n\) we have an embedding \(GL(V_1)\otimes GL(V_2)\xhookrightarrow{} GL(V)\). 

In this section we consider the case when \(n_1\neq n_2\). The case when \(n_1= n_2\) has a different structure and these groups are accounted for in Aschbacher's \(\mathcal{C}_7\) class. 

We now state the main definition of this section.     

\begin{defn}
A group \(G\leq GL(V)\) is a member of \(\mathcal{C}_4\) if \(G=GL(V_1)\otimes GL(V_2)\), where \(V_1\) and \(V_2\) are \(F\)-vector spaces of dimensions \(n_1\) and \(n_2\), such that \(n_1\neq n_2\) and \(n_1n_2=n\). Such groups are isomorphic to \(GL_{n_1}(q)\circ GL_{n_2}(q)\).
\end{defn}

\begin{rem}
\hypertarget{4.2}{Our} definition of \(\mathcal{C}_4\) looks quite different to Aschbacher's original. He explicitly defines a representation which encompasses the embedding \(GL(V_1)\otimes GL(V_2)\xhookrightarrow{} GL(V)\) that we defined above. The domain of his representation is a larger subgroup of \(Aut(S(V_1,f_1))\times Aut(S(V_2,f_2))\) and he has defined the members of \(\mathcal{C}_4\) to be the images of particular subgroups of this domain. The generality of his definition is to account for the different possible forms and the more complicated overgroups of \(I(V,f)\) in \(Aut(S(V,f))\); however, insofar as it relates to \(GL(V)\), there is a perfect correspondence between the members of his definition and ours. 
   
\end{rem}

Though the members of \(\mathcal{C}_4\) are defined on two vector spaces, the general notion of the tensor product stabilisers can be extended to any finite number tensor factors. If \(V_1, ... \,, V_k\) are a collection of \(F\)-vector spaces, we can define \(GL(V_1)\otimes \cdots \otimes GL(V_k)\) as the quotient of \(GL(V_1)\times \cdots \times GL(V_k)\) by the subgroup \(\{(\lambda_1, ... , \lambda_{k-1}, \delta)\:|\:\lambda_i \in F^{\times},\: \delta=(\lambda_1\lambda_2\cdots\lambda_{k-1})^{-1}\}\).

If \(H=GL(V_1)\otimes \cdots \otimes GL(V_k)\), then a subgroup \(G_i\leq GL(V_i)\) can be identified with the subgroup \(1\otimes \cdots \otimes 1 \otimes\, G_i\, \otimes 1 \otimes\cdots \otimes 1\leq H\). For notational convenience, we make use of this identification in the next result; referring to \(G_i\) as a subgroup of \(H\). The statement and proof of this lemma is taken directly from (\cite{kleidman_subgroup_1990}, p.129-131) and it will be of great use to us, not only in the proof of the main theorem, but in numerous results in the subsequent sections. 

\begin{lem}
\hypertarget{4.3}{Let} \(V=\bigotimes_{i=1}^k V_i\) and \(G_i\leq GL(V_i)\), for \(1\leq i\leq k\). If \(V_k\) is an absolutely irreducible \(FG_k\)-module, then:
\begin{enumerate}[label=(\roman*),ref=(\roman*)]
 
\item \(C_{GL(V)}(G_k)=GL(V_1\otimes \cdots \otimes V_{k-1})\otimes 1\).
\item \(N_{GL(V)}(G_k)=GL(V_1\otimes \cdots \otimes V_{k-1})\otimes N_{GL(V_k)}(G_k)\).

\item If \(V_i\) is absolutely irreducible for all \(1 \leq i\leq k\), then \(\bigcap_{i=1}^k  N_{GL(V_i)}(G_i) = \bigotimes_{i=1}^k  N_{GL(V)}(G_i)\).
\end{enumerate}
\end{lem} 
\begin{proof}
We need only consider the case where \(k=2\), as the result follows by simple induction on \(k\). If \(\mathcal{B}_1=\{v_1, ... , v_{n_1}\}\) and \(\mathcal{B}_2=\{w_1, ... , w_{n_2}\}\) are bases for \(V_1\) and \(V_2\) respectively, then \(\mathcal{B}=\{v_i\otimes w_j | 1\leq i \leq n_1, 1\leq j \leq n_2\}\) is a basis for \(V\). If \(g\in G_2\), then with respect to the lexicographical ordering, \([\,g\,]_{\mathcal{B}}\) takes the form:
\[\begin{pmatrix} [\,g\,]_{\mathcal{B}_2} & & \\ & \ddots & \\ & & [\,g\,]_{\mathcal{B}_2}\end{pmatrix}\]
Let \(h\) be arbitrary in \(N_{GL(V)}(G_2)\) and write \([\,h\,]_{\mathcal{B}}\) in the form: 
\[\begin{pmatrix} B_{1,1} & \cdots & B_{1,n_1} \\ \vdots & \ddots & \vdots \\ B_{n_1,1} & \cdots & B_{n_1,n_1}\end{pmatrix}\]
where the \(B_{ij}\) are \(n_2\times n_2\) matrices with entries in \(F\). We observe the following identity.
\[\begin{bmatrix} \,g\, \end{bmatrix}_{\mathcal{B}_2}B_{ij}=B_{ij}\begin{bmatrix}\,h^{-1}gh\,\end{bmatrix}_{\mathcal{B}_2} \label{identity} \tag{eq 4.1}\]

In order to prove part (i), assume \(h\in C_{GL(V)}(G_2)\). Since \(V_2\) is an absolutely irreducible \(FG_2\)-module, Schur's Lemma tells us that each \(B_{ij}\) is a scalar multiple of the identity - call this scalar \(\lambda_{ij}\). Therefore \([\,h\,]_{\mathcal{B}_1}\) is the matrix:
\[\begin{pmatrix} \lambda_{1,1} & \cdots & \lambda_{1,n_1} \\ \vdots & \ddots & \vdots \\ \lambda_{n_1,1} & \cdots & \lambda_{n_1,n_1}\end{pmatrix}\]
The non-singularity of this matrix follows from the non-singularity of \([\,h\,]_{\mathcal{B}}\); the matrix with blocks \(\lambda_{ij}I_{n_2}\). Thus, we can conclude that \(h\in GL(V_1)\). Since it is clear that all of \(GL(V_1)\) centralises \(G_2\), this establishes part (i).

For part (ii) consider the identity \eqref{identity}. By Schur's Lemma each \(B_{ij}\) must be either the \(0\)-matrix or have an inverse. Since \(h\) is non-zero, we must have at least one \(B_{ij}\) that is non-zero - call this non-singular matrix \(B\). We therefore have the following identity. \[B^{-1}\begin{bmatrix} \,g\, \end{bmatrix}_{\mathcal{B}_2}B=\begin{bmatrix}\,h^{-1}gh\,\end{bmatrix}_{\mathcal{B}_2} \text{ for all } g\in G_2 \label{identity2} \tag{eq 4.2}\] 
Now if \(h'\in GL(V)\) such that \([\,h'\,]_{\mathcal{B}}\) is of the form:
\[\begin{pmatrix} B & & \\ & \ddots & \\ & & B\end{pmatrix} \label{identity3} \tag{eq 4.3}\]
then \([\,h'\,]_{\mathcal{B}_2}=B\) and \(h'\) is an element of \(GL(V_2)\). By \eqref{identity2}, \(h'\) normalises \(G_2\) and \(h(h')^{-1}\in C_{GL(V)}(G_2)=GL(V_1)\). Hence \(h \in GL(V_1)N_{GL(V_2)}(G_2)\). Once again, the reverse containment is clear, thus establishing part (ii). Part (iii) is a direct corollary of part (ii).
\end{proof}

\begin{pro}
\hypertarget{4.4}{If} \(V_1\) and \(V_2\) are \(FG\)-modules with \(G\) acting trivially on \(V_2\), then \(V_1\otimes V_2\) is \(FG\)-isomorphic to \(V_1^{\oplus\, dimV_2}\).
\end{pro}
\begin{proof}
Let \(\{v_1,\:...\:,v_d\}\) be a basis of \(V_2\). Recalling \hyperlink{0.3.6}{(0.3.6)}, the result follows from observing the \(FG\)-isomorphisms:
\[V_1\otimes V_2\cong V_1\otimes (\bigoplus_{i=1}^d\, \langle v_i\rangle)\cong \bigoplus_{i=1}^d (V_1\otimes \langle v_i\rangle)\cong V_1^{\oplus d}\]
\end{proof}

\section{Aschbacher Class \(\mathcal{C}_5\) - Subfield Stabilisers}

Let \(\mathcal{B}=\{v_1,\, ...\,,v_n\} \) be a \(F\)-basis of \(V\), let \(k\) be a proper subfield of \(F\) and define \(W\subset V\) to be the \(k\)-span of \(\mathcal{B}\). If \(v=\sum_{i=1}^{n} \lambda_iv_i\) is an arbitrary element of \(V\) and \(g\in GL(W,k)\) such that \(g(v_i)=:w_i\), then we have a natural action of \(GL(W,k)\) on \(V\) defined by \(g\cdot v=\sum_{i=1}^{n} \lambda_iw_i\). This action is faithful since if \(g\cdot v=v\) for all \(v\), then \(g\cdot v_i=v_i\) for all \(i\) and therefore \(g=id_V\).

Thus we get an embedding \(GL(W,k)\xhookrightarrow{}GL(V,F)\). Under this action, an element of \(GL(W,k)\) necessarily stabilises \(W\subset V\). Conversely, if \(g\in GL(V,F)\) and it stabilises \(W\), then it is also an element of the embedded copy of \(GL(W,k)\). Therefore \(GL(W,k)\cong N_{GL(V)}(W)\).   

\begin{rem}
In general, \(N_{GL(V)}(W)\) will not be maximal for two reasons.
\begin{enumerate}[label=(\roman*),ref=(\roman*)]
\item \hypertarget{5.1}{In} a similar manner to the comment above \hyperlink{3.1}{(3.1)}, if \(|F:k|\) is composite, there is a field \(K\) such that \(k< K< F\) and then \(N_{GL(V)}(W)< N_{GL(V)}(U)\), where \(U\) is the \(K\)-span of \(\mathcal{B}\). 
 
\item If \(\lambda \in F\,\backslash\, k\), then \(\lambda \cdot id \notin N_{GL(V)}(W)\). Thus \(N_{GL(V)}(W)< N_{GL(V)}(W)F^{\times}\).
\end{enumerate}
\end{rem}

With this in mind, we state the main definition of the section. 

\begin{defn}
A group \(G\) is a member of \(\mathcal{C}_5\) if \(G=N_{GL(V)}(W)F^{\times}\), where \(W\) is the \(k\)-span of some basis of \(V\) and \(k\) is a subfield of \(F\) such that \(|F:k|=r\) is a prime divisor of \(n\). Such groups are isomorphic to  \(GL_n(q^{1/r})\circ C_{q-1}\).
\end{defn}

\begin{rem}
Aschbacher includes an additional condition for \(G\) to be a member of \(\mathcal{C}_5\); namely that \(W\) is an absolutely irreducible \(kN_{I(V,f)}(W)\)-module. But in our case, \(I(V,f)=GL(V,F)\) and \(C_{GL(W,k)}(N_{GL(V,F)}(W))=C_{GL(W,k)}(GL(W,k))=k^{\times}\). So by \hyperlink{0.3.11}{(0.3.11)}, \(W\) is always an absolutely irreducible \(kN_{GL(V,F)}(W)\)-module and we need not state this condition.
\end{rem}

The next definition states what it means for a module to be \textit{realised} over a subfield. In some literature (e.g. Aschbacher), the terminology \textit{written} or \textit{defined} is preferred. In this definition, we abuse notation slightly by identifying the group \(G\) as a subgroup of \(GL(V,F)\), where \(V\) is a \(FG\)-module.

\begin{defn}
\hypertarget{5.4}{Let} \(k\) be a proper subfield of \(F\). A \(FG\)-module \(V\) is said to be \textit{realised} over \(k\) if there exist a \(F\)-basis \(\mathcal{B}\) of \(V\) such that \(G\leq N_{GL(V)}(W)\), where \(W\) is the \(k\)-span of \(\mathcal{B}\).
\end{defn}

\begin{pro}
If \(k\) is a proper subfield of \(F\) and \(V\) is an absolutely irreducible \(FG\)-module that can be realised over \(k\), then \(V=W^F\), for some absolutely irreducible \(kG\)-module \(W\).
\end{pro}

\begin{proof}
Since \(V\) can be realised over \(k\), we know that \(G\leq N_{GL(V)}(W)\) for some \(kG\)-module \(W\) such that \(V=W^F\). First, we claim that \(W\) is an irreducible \(kG\)-module. To see this, suppose there was some proper, non-trivial \(kG\)-submodule \(U\subset W\). But then \(U^F\) is a proper, non-trivial \(FG\)-submodule of \(W^F=V\), which contradicts the irreducibility of \(V\). 

Next we claim that \(W\) is an absolutely irreducible \(kG\)-module.
Indeed, there exist finite fields \(K\) and \(M\) such that \(K\) contains \(k\) and \(M\) contains both \(F\) and \(K\). Then \(V^M=W^F\otimes M=W^M=W^K\otimes M\) and so any \(FG\)-submodule of \(W^K\) is also a \(FG\)-submodule of \(V^M\), but  by the absolutely irreducibility of \(V\), the latter has no proper, non-trivial \(FG\)-submodules. Thus \(W\) is an absolutely irreducible \(kG\)-module.
\end{proof}

\begin{pro}
\hypertarget{5.6}{If} \(k\) is a subfield of \(F\) and \(V\) is an absolutely irreducible \(FG\)-module that can be realised over \(k\), then \(N_{GL(V,F)}(G)\leq N_{GL(V,F)}(W)F^{\times}\), where \(W\) is the \(k\)-span of some basis of \(V\).
\end{pro}

\begin{proof}
By the previous proposition, \(V=W^F=W\otimes_k F\), for some absolutely irreducible \(kG\)-module \(W\) and therefore \(W\) is the \(k\)-span of some basis of \(V\). Define \(N=N_{GL(V,F)}(G)\) and let \(h\in N\). Since \(V\) is realised over \(k\), we know that \(GL(V,F)\) embeds into \(GL(V,k)\) and therefore we can consider \(G\) as a subgroup of the latter and \(h\) as an element of \(N_{GL(V,k)}(G)\). Thus \(N\leq N_{GL(V,k)}(G)\). By \hyperlink{4.3}{(4.3.ii)}, we observe: 
\[N_{GL(V,k)}(G)=GL(F,k)\otimes N_{GL(W,k)}(G)\]

\noindent and if \(N_0:=(GL(F,k)\otimes 1)\cap GL(V,F)\), we have:
\[N= N_0\otimes N_{GL(W,k)}(G)\]
Since \(G\) is contained in the right-hand tensor factor, \(N_0\) is contained in \(C_{GL(V,F)}(G)\), which is equal to \(F^{\times}\) by \hyperlink{0.3.11}{(0.3.11)}. Therefore \(N\leq F^{\times}\otimes N_{GL(W,k)}(G)\leq F^{\times}N_{GL(V,F)}(W)\).
\end{proof}

\section{Aschbacher Class \(\mathcal{C}_6\) - Symplectic-type Group Stabilisers}
Throughout this section, let \(r\) be a prime number.
\subsection*{Extraspecial and Symplectic-type \(r\)-groups}

\begin{defn}
A \(r\)-group \(R\) is said to be \textit{extraspecial} if the following hold.
\begin{enumerate}[label=(\roman*),ref=(\roman*)]
\item \(|R|=r^{1+2m}\), for some positive integer \(m\).
\item \(|Z(R)|=r\).
\item \(R/Z(R)\cong \mathbb{Z}_r^{2m}\).

\end{enumerate}
\end{defn}

Part (iii) of this definition highlights that if \(R\) is an extraspecial group, then \(R/Z(R)\) can be viewed as a \(2m\)-dimensional vector space over \(\mathbb{Z}_r\). Furthermore, the \textit{commutator map} \(R/Z(R)\times R/Z(R) \longrightarrow R'\,;\, (gZ(R)\,,\, hZ(R))\mapsto [g,h]\) functions as a symplectic form on \(R/Z(R)\). Indeed, combining (ii) and (iii) of the definition, we deduce that \(R'=Z(R)\cong\mathbb{Z}_r\). The commutator map is non-degenerate since the only elements of \(R\) commuting with all others are, by definition, in the centre, and it is alternating since every element commutes with itself. 

The classification of such forms in \hyperlink{0.1.13}{(0.1.13)} allows us to classify all extraspecial groups (see \cite{cameron_notes_2000}, p.86-88 for more details). We find that there are only two isomorphism types for any given order, denoted \(r^{1+2m}_{+}\) and \(r^{1+2m}_{-}\). If \(r\) is odd, the former can be distinguished as the extraspecial group of order \(r^{1+2m}\) in which all non-identity elements are of order \(r\). The latter contains an element of order \(r^2\). For the purposes of this paper, when \(r\) is odd, we are concerned only with the group \(r^{1+2m}_{+}\) and from now on we will denote it by \(r^{1+2m}\). If \(r=2\), then \(2^{1+2m}_{+}\) is distinguished from \(2^{1+2m}_{-}\) as the extraspecial group of order \(2^{1+2m}\) containing more elements of order \(2\).

The simplest examples of these extraspecial groups is when \(m=1\). If \(r=2\), then \(2^{1+2}_{+}\) and \(2^{1+2}_{-}\) are distinct non-abelian groups of order \(8\). Therefore, one must be isomorphic to \(D_8\) and the other to \(Q_8\). It is clear that \(D_8\) has more elements of order \(2\), so we deduce that \(2^{1+2}_{+}\cong D_8 \) and \( 2^{1+2}_{-} \cong Q_8\). If \(r\) is odd, then \(r^{1+2}\) is the group \hypertarget{pres}{presented}: \[\langle x,y,z\,|\, x^r=y^r=z^r=[x,z]=[y,z]=e\,,\,[x,y]=z,\,\rangle\]

Phillip Hall proved the following result which shows that all extraspecial groups can be built up from the case where \(m=1\), a proof of which can be found in (\cite{suzuki_group_1982}, p.69-71).

\begin{lem}
\hypertarget{6.2}{If} \(R\) is an extraspecial group of order \(r^{1+2m}\), then:
\begin{enumerate}[label=(\roman*),ref=(\roman*)]
\item If \(r\) is odd and \(R=r^{1+2m}\), then \(R\) is a central product of \(m\) copies of \(r^{1+2}\).
\item If \(R=2^{1+2m}_{+}\), then \(R\) is isomorphic to a central product of \(m\) copies of \(D_8\).
\item If \(R=2^{1+2m}_{-}\), then \(R\) is isomorphic to a central product of \(m-1\) copies of \(D_8\) and one copy of \(Q_8\).
\end{enumerate}
\end{lem}

We now move on to look at a related type of \(r\)-group.

\begin{defn}
\hypertarget{6.3}{A} \(r\)-group is said to be of \textit{symplectic-type} if every characteristic abelian subgroup is cyclic.
\end{defn}

Hall also proved the following result, classifying all symplectic-type \(r\)-groups. See (\cite{suzuki_group_1982}, p.75-79) for a proof.

\begin{lem}
\hypertarget{6.4}{If} \(R\) is a symplectic-type \(r\)-group, then \(R=E\circ S\), where: 
\begin{enumerate}[label=(\roman*),ref=(\roman*)]
\item \(E\) is either trivial or equal to \(r^{1+2m}\), \(2^{1+2m}_{+}\) or \(2^{1+2m}_{-}\).
\item \(S\) is either cyclic (including the trivial group) or \(r=2\) and \(S\) is isomorphic to \(D_{2^n}\), \(Q_{2^n}\) or \(SD_{2^n}\), where \(n\geq 4\).
\end{enumerate}
The subgroups identified in the quotient are \(Z(E)\) and an isomorphic copy in \(S\).  
\end{lem}

\subsection*{Symplectic-type \(r\)-groups of minimal exponent}

We will be particularly interested in symplectic-type \(r\)-groups satisfying a certain minimality condition. 

\begin{defn}
The \textit{exponent} of a finite group \(G\) is the lowest common multiple of the orders of elements in \(G\). 
\end{defn}

When \(r\) is odd, the minimal exponent amongst all symplectic-type \(r\)-groups is \(r\) and the only group satisfying this is \(r^{1+2m}\). When \(r=2\), the minimal exponent amongst symplectic-type \(2\)-groups is \(4\) and there are three groups satisfying this condition: \(2^{1+2m}_{+}, 2^{1+2m}_{-} \text{ and } C_4\circ 2^{1+2m}_{+}\).

It is these four symplectic-type \(r\)-groups of minimal exponent which will be the focus of our discussions for the rest of this section. We proceed by exploring some properties of these groups, beginning with the following isomorphism.

\begin{pro}
\(C_4\circ 2^{1+2m}_{+}\cong C_4\circ 2^{1+2m}_{-}\)
\end{pro}

\begin{proof}
By \hyperlink{6.2}{(6.2)}, it is enough to show that \(C_4\circ D_8\cong C_4\circ Q_8\). The former can be presented \(\langle x, y, c\,|\, x^4=c^4=y^2=e, x^2=c^2, cx=xc, cy=yc, yxy=x^{-1}\rangle\). The subgroup \(\langle x, cy\rangle\) is of order \(8\) since \(x^2=(cy)^2\) and it is non-abelian since these generators do not commute. Both generators are of order \(4\), therefore it must be isomorphic to \(Q_8\). Finally, we can observe that the given presentation is equal to \(\langle  x, cy\rangle\circ \langle c \rangle\), where \(\langle (cy)^2\rangle\) is associated with \(\langle c^2\rangle\).      
\end{proof}

We will therefore refer to \(C_4\circ 2^{1+2m}_{+}\) and \(C_4\circ 2^{1+2m}_{-}\) simply as \(C_4\circ 2^{1+2m}\). Next, we state a condition that allows us to identify when a symplectic-type \(r\)-group is of minimal exponent. 

\begin{pro}
\hypertarget{6.7}{If} \(R\) is a non-abelian symplectic-type \(r\)-group with no proper non-central characteristic subgroup, then it is of minimal exponent.
\end{pro}

\begin{proof}
First, consider the case when \(r\) is odd. By \hyperlink{6.4}{(6.4)}, \(R\cong C_n\circ E\), where \(E\) is either trivial or \(r^{1+2m}\). If \(E\) is trivial, then \(R\) is abelian, so we must have \(R\cong C_n\circ r^{1+2m}\). We note that \(n\) is a multiple of \(r\) and that \(C_r\circ r^{1+2m}\cong r^{1+2m}\). If \(n>r\), then \(r^{1+2m}\) is the subgroup generated by all elements of order \(r\) in \(R\) and it is therefore characteristic. Thus, we must have \(R= r^{1+2m}\).

Now let \(r=2\). By \hyperlink{6.4}{(6.4)}, \(R= S\circ E\), where \(S\) is either cyclic of even order or isomorphic to \(D_{2^n}\), \(Q_{2^n}\) or \(SD_{2^n}\) and \(E\) is trivial or isomorphic to \(2^{1+2m}_{+}\) or \(2^{1+2m}_{-}\). Suppose \(E\) is trivial, then since \(R\) is non-abelian \(S\) cannot be cyclic. Each of the other possible groups has a unique index \(2\) (non-central) subgroup, which is therefore characteristic in \(R\). Thus \(E\) must be an extraspecial \(2\)-group.

Let \(S\) be any of the possible groups other than a cyclic group and consider the inclusion map \(\phi: S \longrightarrow R\). Let \(y\) be a generator of the index \(2\) subgroup of \(S\). Then the generators of \(\langle \phi(y) \rangle\) are the unique elements of their given order that are also centralised by \(E\). Therefore \(\langle \phi(y) \rangle\) is characteristic in \(R\) and thus \(R=C_n\circ 2^{1+2m}_{\pm}\). If \(n>4\), then the subgroup generated by all elements of order \(2\) and \(4\) is: a) characteristic; by reason of element orders b) non central; since it contains a copy of the extraspecial \(2\)-group and c) properly contained in \(R\); since the element \(ge\), where \(g\) is a generator of \(C_n\) and \(e\) is the identity element of the extraspecial group, is of order greater than \(4\). Noting that \(C_2\circ 2^{1+2m}_{\pm}\cong 2^{1+2m}_{\pm}\), the result follows.
\end{proof}

\subsection*{Representation theory of minimal exponent symplectic-type groups}

We will now turn our attention to discussing the representation theory of these groups. In order to establish which fields we can find irreducible representations over, we begin by stating the following well-known result. See (\cite{serre_linear_1996}, p.64-65) for a proof.

\begin{pro}
\hypertarget{6.8}{If} \(R\) is a \(r\)-group and \(F\) is a field of characteristic \(r\), then any irreducible \(F\)-representation \(\rho: R\longrightarrow GL(V)\) is trivial.
\end{pro}

Therefore, if we want to construct a non-trivial irreducible representation for a symplectic-type \(r\)-group, we must do so over a field with characteristic \(p\neq r\). Note that this puts us in the situation discussed above \hyperlink{0.3.13}{(0.3.13)}. 

We will construct these representations shortly and we shall see that, in order to do so, the field must contain all \(k^{th}\) roots of unity, where \(k:=|Z(R)|\). This places further restrictions on the size of the field, since \(F\) contains the \(k^{th}\) roots of unity if and only if \(k\) divides \(p^e-1\). Therefore, from this point on, let \(F\) be a field of order \(p^e\), where \(p\) is a prime not equal to \(r\), such that \(p^e \equiv 1\: (\,mod\, |Z(R)|\,)\).

We can deduce how many irreducible representations a symplectic-type \(r\)-group of minimal exponent has by some nice counting arguments. We divide this task into separate two cases; first when \(R\) is one of the three extraspecial groups, secondly when \(R\cong C_4\circ 2^{1+2m}\).

\begin{pro}
If \(R\) is an extraspecial symplectic-type \(r\)-group of minimal exponent, then \(R\) has \(r^{2m}+r-1\) irreducible representations; \(r-1\) of which are \(r^m\)-dimensional and the rest are \(1\)-dimensional.
\end{pro}

\begin{proof}
Let \(x\) be a non-central element of \(R\). On the one hand, since \(R/Z(R)\) is abelian, the conjugacy class of \(x\) must be contained in \(xZ(R)\). On the other hand, the size of the conjugacy class is a positive power of \(r\) and therefore it must be exactly \(r\). Hence the conjugacy class is all of \(xZ(R)\). So, we get \(r^{2m}-1\) conjugacy classes from the non-central elements and one from each central element. Thus \(R\) has a total of \(r^{2m}+r-1\) conjugacy classes and by \hyperlink{0.3.13}{(0.3.13)}, \(R\) has the same number of irreducible representations. 

By \hyperlink{0.3.14}{(0.3.14)}, \(|R/R'|=r^{2m}\) of these irreducibles are \(1\)-dimensional and by \hyperlink{0.3.15}{(0.3.15)}, the sum of the remaining \(r-1\) irreducibles degree's squared is \(r^{1+2m}-r^{2m}=r^{2m}(r-1)\). Suppose one of these irreducibles has degree squared greater than \(r^{2m}\). Since the degree is a power of \(r\), it must be at least \(r^{1+2m}\); however, this is greater than \(r^{2m}(r-1)\), so it is not possible. Therefore, all \(r-1\) irreducibles must have degree squared equal to \(r^{2m}\); that is, they have degree \(r^{m}\).
\end{proof}

\begin{pro}
If \(R\cong C_4\circ 2^{1+2m}\), then \(R\) has \(2^{2m+1}+2\) irreducible representations; two of which are \(2^m\)-dimensional and the rest are \(1\)-dimensional.
\end{pro}

\begin{proof}
We begin by noting that \(R'\cong C_2\). If \(x\in R\backslash R'\), then the conjugacy class of \(x\) is contained in \(xR'\), since \(R/R'\) is abelian. Hence, conjugacy classes of non-central elements must be of size \(2\). Therefore \(R\) has \(2^{2m+1}+2\) conjugacy classes and irreducible representations. 

Of these representations, \(|R/R'|=2^{2m+1}\) are \(1\)-dimensional, so by the same argument as the extraspecial case, the sum of degrees squared of the two remaining irreducibles are \(2^{2m+2}-2^{2m+1}=2^{2m}\cdot 2\) and this means that they are of degree \(2^m\) .
\end{proof}

The \(1\)-dimensional representations of these groups are just those corresponding to the elementary abelian group \(R/R'\). It is the \(r^m\)-dimensional irreducibles that are of interest to us and it is simple enough to construct them. We need only construct such a representation for \(2_{\pm}^{1+2}\) and \(r^{1+2}\), since by \hyperlink{6.2}{(6.2)}, we may then tensor \(m\) irreducibles of the corresponding groups to get the desired representation for \(2_{\pm}^{1+2m}\) and \(r^{1+2m}\) (and for \(C_4\circ 2^{1+2m}\), we tensor with an additional irreducible of \(C_4\)). 

First, consider the case when \(R=2_{\pm}^{1+2}\). We have already noted that the two groups are isomorphic to \(D_8\) and \(Q_8\) and the reader should be familiar with the (absolutely) irreducible \(2\)-dimensional  representations of these. In the case of \(D_8=\langle x,y\rangle\), where \(x\) is of order \(4\) and \(y\) is of order \(2\):
\[\rho(x)=\begin{pmatrix} 0 & -1 \\ 1 & 0\end{pmatrix}\:\text{, }\: \rho(y)=\begin{pmatrix} 1 & 0 \\ 0 & -1\end{pmatrix}\]
In the case of \(Q_8=\langle i,j\rangle\), where \(i\) and \(j\) are of order \(4\): 
\[\rho(i)=\begin{pmatrix} \sqrt{-1} & 0 \\ 0 & -\sqrt{-1}\end{pmatrix}\:\text{, }\: \rho(j)=\begin{pmatrix} 0 & -1 \\ 1 & 0\end{pmatrix}\]

Tensoring the first of these representations \(m\) times gives us an absolutely irreducible \(2^m\)-dimensional representation of \(2_{+}^{1+2m}\). Tensoring \(m-1\) copies of the former and one of the latter, gives us the same for \(2_{-}^{1+2m}\). From these, we derive the two \(2^m\)-dimensional irreducibles for \(C_4\circ 2^{1+2m}\) by taking the above representation of \(2_{+}^{1+2m}\) and tensoring it with one of the two faithful irreducible representations of \(C_4\), each one producing a distinct absolutely irreducible \(2^m\)-dimensional representation. All of these representations are faithful since the individual factors are. 

Next we consider the case when \(r\) is odd and recall the presentation of \(R=r^{1+2}\) we stated at the beginning of this section. Let \(\lambda\) be a primitive \(r^{th}\) root of unity and consider the \(r\)-dimensional representation:
\[\rho_{\lambda}(x)=\begin{pmatrix} 1 & & & & \\ &\lambda & & & \\ & & \lambda^2 & & \\ & & &\ddots & \\ & & & &\lambda^{r-1} \end{pmatrix}\:\text{, }\: \rho_{\lambda}(y)=\begin{pmatrix} 0 & 0 & \cdots & 0 & 1 \\  1 & 0 &  \cdots & 0 & 0 \\ 0 & 1 & \ddots &\vdots& 0 \\ \vdots & \ddots & \ddots & 0 & \vdots\\ 0 & \cdots & 0 & 1 & 0 \end{pmatrix}\]
The relation \([x,y]=z\) determines that \(\rho_{\lambda}(z)=\lambda I_n\). The kernel of \(\rho_{\lambda}\) is a normal subgroup of \(R\) and thus has order \(1,r\) or \(r^2\). If the kernel is non-trivial then the quotient group is abelian and therefore the kernel contains \(R'=Z(R)\), which is a contradiction. Therefore \(\rho_{\lambda}\) is faithful. Additionally, the representation is irreducible. Indeed, any subspace \(W\subset V\) that is stabilised by \(\rho_{\lambda}(x)\), is a sum of its eigenspaces. Since \(\rho_{\lambda}(x)\) has \(r\) distinct eigenvalues, \(W\) must contain at least one standard basis vector \(e_i\). If \(W\) is also stabilised by the cyclic shift matrix \(\rho_{\lambda}(y)\), it then contains all standard basis vectors. Therefore \(W=V\) and \(\rho_{\lambda}(R)\) is irreducible.

Furthermore, since it has distinct eigenvalues, any matrix commuting with \(\rho_{\lambda}(x)\) must be diagonal and any matrix commuting with \(\rho_{\lambda}(y)\) must have the same entries along the diagonal. Therefore \(C_{GL(V)}(\rho_{\lambda}(R))\cong F^{\times}\) and thus by \hyperlink{0.3.11}{(0.3.11)}, \(\rho_{\lambda}(R)\) is absolutely irreducible. 

Exchanging \(\lambda\) for the other primitive \(r^{th}\) roots of unity, we get a set of \(r-1\) absolutely irreducible faithful \(r\)-dimensional representations \(\{\rho_{\lambda}, \rho_{\lambda^2}, . . . , \rho_{\lambda^{r-1}}\}\). Note that these representations are wholly determined by where they send a fixed generator of the centre i.e they are determined by \(\lambda\) (this same point can be said about the \(2^m\)-dimensional irreducibles of \(R\cong C_4\circ 2^{1+2m}\) and it is trivially true for the extraspecial \(2\)-groups, which only has one such irreducible).
Once again, tensoring one these representations \(m\) times we get a faithful absolutely irreducible \(r^m\)-dimensional representation for the group \(r^{1+2m}\). This discussion is summarised in the following result.

\begin{lem}
Let \(R\) be a symplectic-type \(r\)-group of minimal exponent and let \(F\) be a field of order \(p^e\), where \(p\) is a prime distinct from \(r\) and \(e\) is an integer such that \(p^e \equiv 1\: (\,mod\, |Z(R)|\,)\). 
\begin{enumerate}[label=(\roman*),ref=(\roman*)]
\item If \(R\) is an extraspecial \(r\)-group, then it has \(r-1\) absolutely irreducible faithful \(r^m\)-dimensional representations determined (up to equivalence) by where they send a fixed generator of the centre. 
\item If \(R\cong C_4\circ 2^{1+2m}\), then it has two absolutely irreducible faithful \(2^m\)-dimensional representations determined (up to equivalence) by where they send a fixed generator of the centre. 

\end{enumerate}

\end{lem}

Thus, if \(F\) is a field with the aforementioned restrictions and \(V\) is a \(r^m\)-dimensional \(F\)-vector space, we can embed such groups into \(GL(V)\) via these representations. We assume these conditions on \(F\) and \(V\) for the remainder of the section. 

\subsection*{Automorphisms of minimal exponent symplectic-type groups}

In the lead up to the main definition of this section, we must explore the structure of the automorphism groups of the symplectic-type \(r\)-groups of minimal exponent.

Using the shorthand \(A:=Aut(R)\), we now state the structure of \(C_{A}(Z(R))\), the group of automorphisms that fix every element in the centre of \(R\). The reason for our interest in this particular group of automorphsims will become apparent shortly.

\begin{lem}
\hypertarget{6.12}{Let} \(R\) be a symplectic-type \(r\)-group of minimal exponent. The structure of \(C_{A}(Z(R))\) is described in the table below.

\begin{center}

\begin{tabular}{ |p{3.5cm}||p{3.5cm}|}
 \hline
 \multicolumn{2}{|c|}{ \textbf{Table 6.12}} \\
 \hline
 \(R\)&  \(C_{A}(Z(R))\)\\
 \hline
\(r^{1+2m}\)   &    \(r^{2m}.Sp_{2m}(r)\)\\
 \(2^{1+2m}_{+}\)&  \(2^{2m}.O^{+}_{2m}(2)\)\\
 \(2^{1+2m}_{-}\) &  \(2^{2m}.O^{-}_{2m}(2)\)\\
\(C_4\circ 2^{1+2m}\)  & \(2^{2m}.Sp_{2m}(2)\)\\
 \hline
\end{tabular}
\end{center}

\end{lem}

\begin{proof}
A full proof for the first three rows of the table are proved in \cite{winter_automorphism_1972}, the last row is justified in (\cite{griess_automorphisms_1973}, p.403-404).
\end{proof}

\begin{pro}
\hypertarget{6.13}{Let} \(R\) be a symplectic-type \(r\)-group of minimal exponent. If \(\rho\) is a \(r^m\)-dimensional irreducible representation of \(R\), then \(Aut_{GL(V)}(\rho(R))\cong C_A(Z(R))\).
\end{pro}

\begin{proof}
We begin by noting that if \(\alpha\in Aut(R)\) and \(\rho\) is a faithful irreducible representation of \(R\), then \(\rho \circ \alpha\) is also a faithful irreducible representation. Since the \(r^m\)-dimensional irreducibles are determined by where they send a generator of the centre, if \(\alpha\in C_A(Z(R))\), then \(\rho \circ \alpha\) must be equivalent to \(\rho\). So there exists \(g\in GL(V)\) such that \(\rho \circ \alpha=g\rho g^{-1}\) and if \(r,r'\in R\) such that \(\alpha(r)=r'\), we have:
\[g\circ \rho(r)\circ g^{-1}=\rho\circ \alpha (r)=\rho(r')\]

Since \(\rho\) is faithful, \(g\) induces a non-trivial automorphism on \(\rho(R)\) if and only \(g\notin C_{GL(V)}(\rho(R))\), which by the above, happens if and only if \(\alpha\neq 1\). Thus we have an injective map from \(C_A(Z(R))\) into \(Aut_{GL(V)}(\rho(R))\).

Let \(\varphi_g\in Aut_{GL(V)}(\rho(R))\), then \(\rho^{-1}\circ \varphi_g\circ\rho\) is a map from \(R\longrightarrow R\). This map fixes a central element \(z\), since \(\rho(z)\) is a scalar map. Furthermore, the faithfulness of \(\rho\) determines that this map is an isomorphism that is unique for unique elements of \(Aut_{GL(V)}(\rho(R))\). Thus, we have an injection from \(Aut_{GL(V)}(\rho(R))\) to \(C_A(Z(R))\) and the result follows.    
\end{proof}

The next proposition ascertains conditions for when the \(r^m\)-dimensional irreducible representations of \(R\) fix a classical form on \(V\).

\begin{pro}
\hypertarget{6.14}{Let} \(R\) be a symplectic-type \(r\)-group of minimal exponent. If \(\rho:R\longrightarrow GL_n(p^e)\) is a \(r^m\)-dimensional irreducible representation, then:
\begin{enumerate}[label=(\roman*),ref=(\roman*)]
 
\item\(\rho(R)\) fixes a symplectic or non-degenerate symmetric bilinear form if and only if \(R\cong 2_{\pm}^{1+2m}\). 
\item \(\rho(R)\) fixes a non-degenerate unitary form if and only if  \(e\) is even.
\end{enumerate}
\end{pro}

\begin{proof}
We know that \(\rho\) is faithful and it is determined by where it sends a fixed generator \(z\in Z(R)\). So if \(\rho(z)=\lambda\cdot I_n\), by definition of the dual representation, we observe that: 
\[ \rho \text{ is equivalent to } \rho^{\ast} \iff \lambda = \lambda^{-1} \iff |Z(R)|=2.\]

Out of the symplectic type \(r\)-groups of minimal exponent, the right-hand side occurs if and only if \(R\cong 2_{\pm}^{1+2m}\). By \hyperlink{0.3.19}{(0.3.19)}, the left-hand side occurs if and only if \(\rho(R)\) fixes a non-degenerate symplectic or symmetric bilinear form, thus establishing (i). 

For part (ii), the only if statement follows immediately from \hyperlink{0.3.19}{(0.3.19)}. Furthermore, in view of this result, we need only show that \(\rho^{\theta} \text{ is equivalent to } \rho^{\ast}\) if and only if \(e\) is even, where \(\theta\) is the field automorphism of order two. However, by the congruence conditions that we've established on the size of the field, we know: 
\[ \text{ \(e\) is even} \iff \lambda^{p^{e/2}} = \lambda^{-1} \iff \rho^{\theta} \text{ is equivalent to } \rho^{\ast}\]

\noindent thus establishing (ii).
\end{proof}

If \(R\) is embedded in \(GL(V)\) and it does fix a classical form \(f\), we would like to determine when an element of \(Aut_{GL(V)}(R))\) is induced by an element of \(\Delta(V,f)\).

\begin{pro}
\hypertarget{6.15}{Let} \(R\) be a symplectic-type \(r\)-group of minimal exponent that acts absolutely irreducibly on \(V\). If \(R\) fixes a symplectic, unitary or non-degenerate symmetric bilinear form \(f\), then \(Aut_{\Delta(V,f)}(R)=Aut_{GL(V)}(R)\). 
\end{pro}

\begin{proof}
The inclusion from left to right is immediate and the inclusion from right to left is a result of \hyperlink{0.3.21}{(0.3.21)}.
\end{proof}

We are now ready to define the members of the sixth Aschbacher class.

\begin{defn}
A subgroup \(G\leq GL(V)\) is a member of \(\mathcal{C}_6\) if \(G=N_{GL(V)}(R)\), where \(R\) is a symplectic-type \(r\)-group of minimal exponent not isomorphic to \(D_8\) or \(2_{\pm}^{1+2m}\), for \(m>1\). In addition, we must have that  \(dim\,V=r^m\) and \(F=\mathbb{F}_{p^e}\), where \(e\) is both odd and the smallest integer such that \(p^e\equiv 1\: (mod\, |Z(R)|)\).
\end{defn}

\begin{rem}
The additional restrictions on \(R\) and \(F\) are to avoid overlap with other Aschbacher classes and are justified below. 

\begin{enumerate}[label=(\roman*),ref=(\roman*)]
 
\item If \(R=2_{\pm}^{1+2m}\), then by \hyperlink{6.14}{(6.14)}, \(R\) fixes a symplectic or non-degenerate symmetric bilinear form \(f\), and by \hyperlink{0.3.21}{(0.3.21)}, \(N_{GL(V)}(R)\leq \Delta(V,f)\). If \(m>1\), then the vector space on which this form acts has dimension greater than two and, as we shall see in section eight, such groups are contained in members of \(\mathcal{C}_8\).

\item If \(V\) is \(2\)-dimensional, we can define a non-degenerate symmetric bilinear form \(f\) on an orthogonal basis \(\{v_1,v_2\}\) by \(f(v_1,v_1)=f(v_2,v_2)=1\) and the \(2\)-dimensional representation of \(D_8\) that we defined earlier fixes \(f\). Thus, when \(R\cong D_8\), we have the containment \(R \leq I(V,f)\). By \hyperlink{0.3.21}{(0.3.21)}, \(N_{GL(V)}(R)\leq \Delta(V,f)\), thus it can never be maximal in \(GL(V)\), explaining its exclusion from this Aschbacher class. However, the case when \(R\cong Q_8\) is not excluded. In this case, \(R\) stabilises a symplectic form \(f\) on a \(2\)-dimensional vector space and due to the isomorphism \hyperlink{0.2.10}{(0.2.10.iii)}, the isometry group of \(f\) can be indeed be maximal.   

\item  \hypertarget{6.17.iii}{We} have already established that it is necessary for \(p\) and \(e\) to be such that \(p^e \equiv 1\:(\,mod\,|Z(R)|\,)\). Suppose, however, that there is a smaller integer \(d\) satisfying this condition. The embedding of \(R\) into \(GL(V,\mathbb{F}_{p^e})\) can be realised over the subfield \(\mathbb{F}_{p^d}\) and therefore by \hyperlink{5.6}{(5.6)}, \(R\) is contained in a member of \(\mathcal{C}_5\). Thus, we require that \(e\) is the smallest such integer.
\end{enumerate}

\end{rem}

The structure of \(N_{GL(V)}(R)\) follows immediately from \hyperlink{6.12}{(6.12)}, \hyperlink{6.13}{(6.13)} and \hyperlink{0.3.11}{(0.3.11)}.

\begin{pro}
For the groups \(R\) satisfying the definition above, the structure of \(N_{GL(V)}(R)\) is as follows.

\begin{enumerate}[label=(\roman*),ref=(\roman*)]
\item If \(R\cong r^{1+2m}\), then \(N_{GL(V)}(R)\cong (C_{q-1}\circ r^{1+2m}).Sp_{2m}(r)\) 
\item If \(R\cong C_4\circ 2^{1+2m}\), then \(N_{GL(V)}(R)\cong (C_{q-1}\circ 2^{1+2m}).Sp_{2m}(2)\) 
\item If \(R\cong Q_8\), then \(N_{GL(V)}(R)\cong (C_{q-1}\circ Q_8).O^{-}_{2}(2)\) 
\end{enumerate}

\end{pro}

\section{Aschbacher Class \(\mathcal{C}_7\) - Wreathed Tensor Product Stabilisers}

In section 4, we discussed the stabiliser of a tensor product where the dimensions of each factor were different. In this section, we will discuss the case when the tensor product factors are of the same dimension. We informally refer to these groups as \textit{wreathed tensor product stabilisers}, as the structure of the group is (effectively) a wreath product. This condition on the dimension also means that, unlike the \(\mathcal{C}_4\) class, we will need to consider when there are any finite number of factors.

Let \(V_1, ... \,, V_k\) be \(m\)-dimensional \(F\)-vector spaces. If \(\eta_j:V_1\longrightarrow V_j\) is a fixed \(F\)-isomorphism, for all \(1\leq j\leq k\), then \(V:=V_1\otimes \cdots \otimes V_k\) is spanned by the elements \(\eta_1(v_1)\otimes \cdots \otimes \eta_k(v_k)\), where the \(v_i\) range across the elements of \(V_1\). By defining isomorphisms \(\alpha_j:GL(V_1)\longrightarrow GL(V_j)\) such that  \(\alpha_j(g)(\eta_j(v))=\eta_j(g(v))\), for \(g\in GL(V_1)\) and \(v\in V_1\), we can observe a faithful action of \(GL(V_1)\otimes \cdots \otimes GL(V_k)\) on \(V\) defined by:
\[(h_1, ... , h_k)\cdot (\eta_1(v_1)\otimes \cdots \otimes \eta_k(v_k))= \eta_1(g_1(v_1))\otimes \cdots \otimes \eta_k(g_k(v_k))\]

\noindent where \(h_j\in GL(V_j)\) and \(g_j\in GL(V_1)\) such that \(\alpha_j(g_j)=h_j\in GL(V_j)\). So, as in the \(\mathcal{C}_4\) case, \(GL(V)\) contains the tensor product stabiliser \(GL(V_1)\otimes\, \cdots\, \otimes GL(V_k)\). However, unlike the \(\mathcal{C}_4\) case, this is not maximal. This leads us to the definition of the seventh Aschbacher class.

\begin{defn}
A group \(G\leq GL(V)\) is a member of \(\mathcal{C}_7\) if \(G=N_{GL(V)}(GL(V_1)\otimes \cdots \otimes GL(V_k))\), where \(V=V_1\otimes \cdots \otimes V_k\) and the \(V_i\) are \(m\)-dimensional \(F\)-vector spaces, with \(m>2\) and \(k>1\). Such groups are isomorphic to \(GL_m(q)\circ\cdots \circ GL_m(q)\rtimes S_k\), but the neatest way to state the structure is to observe that \(G/Z(G)\) is isomorphic to \(PGL_m(q)\wr S_k\).
\end{defn}

\begin{rem} 
We make the following observations about the definition above.
\begin{enumerate}[label=(\roman*),ref=(\roman*)]
\item The comment in \hyperlink{4.2}{(4.2)}, which we made in regards to the definition of \(\mathcal{C}_4\), suffices to explain the divergence between Aschbacher's definition of \(\mathcal{C}_7\) and ours (though the domain of the representation in Aschbacher's \(\mathcal{C}_7\) is larger than that of the \(\mathcal{C}_4\) case). 

\item The additional condition \(m>2\) is given to avoid overlap with the \(\mathcal{C}_8\) class. If \(m=2\), then by \hyperlink{0.2.10}{(0.2.10.iii)}, there exists symplectic forms \(f_i\) such that \(SL(V_i)\cong I(V_i,f_i)\), for \(1\leq i\leq k\). By the construction in \hyperlink{0.1.3}{(0.1.3)}, we can define the form \(f=f_1\otimes \cdots \otimes f_k\) on the vector space \(V_1\otimes \cdots \otimes V_k\). This form \(f\) is symplectic when \(k\) is odd and non-degenerate symmetric bilinear when \(k\) is even. Thus \(SL(V_1)\otimes \cdots \otimes SL(V_k)\leq I(V_1,f_1)\otimes \cdots \otimes I(V_k,f_k)\leq I(V,f)\). 
We will show in the next proposition that \(S:=SL(V_1)\otimes\, \cdots\, \otimes SL(V_k)\) is characteristic in \(G:=GL(V_1)\otimes\, \cdots\, \otimes GL(V_k)\), thus \(N_{GL(V)}(G)\leq N_{GL(V)}(S)\) and by \hyperlink{0.3.21}{(0.3.21)}, the latter is contained in \(\Delta(V,f)\), which is a member of \(\mathcal{C}_8\).

\end{enumerate}

\end{rem}

In the following two lemmas, we justify our claim about the structure of the \(\mathcal{C}_7\) members. We will again make use of the convention (noted above \hyperlink{4.3}{(4.3)}) of identifying a subgroup \(G_i\leq GL(V_i)\) with the subgroup \(1\otimes \cdots \otimes 1 \otimes\, G_i\, \otimes 1 \otimes\cdots \otimes 1\).

\begin{pro}
Let \(m>2\) and \(k>1\). If \(V_1, ... \,, V_k\) are \(m\)-dimensional \(F\)-vector spaces, then \(E(GL(V_1)\otimes\, \cdots\, \otimes GL(V_k))=SL(V_1)\otimes\, \cdots\, \otimes SL(V_k)\).
\end{pro}

\begin{proof}
Let \(G=GL(V_1)\otimes\, \cdots\, \otimes GL(V_k)\) and \(S=SL(V_1)\otimes\, \cdots\, \otimes SL(V_k)\). In view of \hyperlink{0.2.21}{(0.2.21)}, \(SL(V_i)\) is a component of \(G\), for all \(1\leq i\leq k\), and therefore \(S\leq E(G)\). What's more \(S\) is normal in \(E(G)\), since it is normal in \(G\). Observing that the quotient \(G/S\) is isomorphic to \(F^{\times}\otimes \cdots \otimes F^{\times}\), we deduce that the quotient \(E(G)/S\) is abelian. Therefore, \(E(G)\) has no further components and the result follows.
\end{proof}

With this lemma we can compute the structure of a member of \(\mathcal{C}_7\). To ease notation in the proof, we keep the shorthand \(G=GL(V_1)\otimes\, \cdots\, \otimes GL(V_k)\) and define \(N=N_{GL(V)}(G)\).

\begin{pro}
\hypertarget{7.4}{Let} \(m>2\) and \(k>1\). If \(V_1, ... \,, V_k\) are \(m\)-dimensional \(F\)-vector spaces and \(V=V_1\otimes \cdots \otimes V_k\), then \(N_{GL(V)}(GL(V_1)\otimes\, \cdots\, \otimes GL(V_k))=(GL(V_1)\otimes\, \cdots\, \otimes GL(V_k))\rtimes S_k\).
\end{pro} 

\begin{proof}
Let \(\sigma\) be an element of the symmetric group \(S_k\). We can define a homomorphism \(\phi: S_k\longrightarrow N\) by: 
\[\phi(\sigma)(\eta_1(v_1)\otimes \cdots \otimes \eta_k(v_k))=\eta_1(v_{\sigma^{-1}(1)})\otimes \cdots \otimes \eta_k(v_{\sigma^{-1}(k)}))\]

Evidently this homomorphism is injective, thus we have an embedding \(S_k\xhookrightarrow{} N\). Furthermore, if \((h_1,...,h_k)\in G\) and \(g_j\in GL(V_1)\) such that \(\alpha_j(g_j)=h_j\) and we define \(w_j:=g_j(v_j)\) for all \(1\leq j\leq k\), then:
\[\phi(\sigma) ((h_1, ... , h_k) (\eta_1(v_1)\otimes \cdots \otimes \eta_k(v_k)))= \phi(\sigma)(\eta_1(w_1)\otimes \cdots \otimes \eta_k(w_k))\]
\[=\eta_1(w_{\sigma^{-1}(1)})\otimes \cdots \otimes \eta_k(w_{\sigma^{-1}(k)})=\eta_1(g_{\sigma^{-1}(1)}(v_{\sigma^{-1}(1)}))\otimes \cdots \otimes \eta_k(g_{\sigma^{-1}(k)}(v_{\sigma^{-1}(k)}))\]
\[=(h_{\sigma^{-1}(1)}, ... , h_{\sigma^{-1}(k)})(\phi(\sigma)(\eta_1(v_1)\otimes \cdots \otimes \eta_k(v_k)))\]

Thus \(\phi(\sigma)(h_1, ... , h_k)\phi(\sigma)^{-1}= (h_{\sigma^{-1}(1)}, ... , h_{\sigma^{-1}(k)})\) i.e. \(\phi(S_k)\) acts by permuting coordinates of \(G\) and therefore the group \(G \rtimes S_k\), where \(S_k\) acts on \(G\) via \(\phi\), is a subgroup of \(N\).

To show the reverse containment, let \(h\in N\). By the previous proposition \(E(G)=SL(V_1)\otimes\, \cdots\, \otimes SL(V_k)\) is characteristic in \(G\) and thus \(h\) acts on \(E(G)\) by permuting its factors. Therefore, there exists some \(\sigma\in S_k\) such that \(\phi(\sigma)h\in \bigcap_{j=1}^k N_{GL(V)}(SL(V_j))\). By \hyperlink{0.3.12}{(0.3.12)}, \(SL(V_j)\) is absolutely irreducible on \(V_j\) and therefore we can apply \hyperlink{4.3}{(4.3.iii)} to conclude that \(\phi(\sigma)h\in N_{GL(V_1)}(SL(V_1))\otimes \cdots \otimes N_{GL(V_k)}(SL(V_k))=G\). Thus \(h\in G\rtimes S_k\) and the result follows.   
\end{proof}

We conclude this section with two results about central products and layers that will be of use to us in \hyperlink{lemma 7}{Lemma 7} in proof of the main theorem.

\begin{pro}
\hypertarget{7.5}{Let} \(L\) be a group and \(\varphi:L\longrightarrow L/Z(L)\) the natural quotient map. If \(\varphi(L)=Y_1 \times \cdots \times Y_k\) is a product of non-abelian simple groups, then \(E(L)=Q_1\circ \cdots \circ Q_k\), where \(\varphi(Q_i)=Y_i\).  
\end{pro}

\begin{proof}
For all \(1\leq i\leq k\), define \(Q_i\) to be a minimal preimage of \(Y_i\). If \(Q_i\) is not perfect, then \(\varphi(Q_i')\) is a proper normal subgroup of \(Y_i\), therefore it must be trivial. So \(Q_i\leq Z(L)\), which implies it is solvable, but this implies that \(Y_i\cong Q_i/(Z(L)\cap Q_i)\) is solvable, which is a contradiction to \(Y_i\) being non-abelian simple. Therefore \(Q_i\) is perfect.

Now \(\varphi([Q_i,Q_j])=[\varphi(Q_i),\varphi(Q_j)]=[Y_i,Y_j]=1\), therefore \([Q_i,Q_j]\leq Z(L)\) and so \([Q_i,Q_j,Q_i]=[Q_i,Q_j,Q_j]=1\). Thus by the \hyperlink{0.2.14}{Three Subgroup Lemma}, \([Q_i,Q_i,Q_j]=[Q_i,Q_j]=1\). So any element of \(Z(Q_i)\) commutes with the rest of \(L\) and therefore \(Z(Q_i)=Z(L)\cap Q_j\). Since \(Y_i\cong Q_i/(Z(L)\cap Q_i)\), we have shown that the \(Q_i\) are quasisimple, and hence are components of \(L\). Furthermore, these are all the components of \(L\), else \(\varphi(L)\) would have additional factors. Hence \(E(L)=Q_1\circ \cdots \circ Q_k\). 
\end{proof}

The final result of this section allows us to identify an absolutely irreducible module of a central product with a tensor product of absolutely irreducible modules of each factor of that central product. The proof uses very similar notation and argumentation to that of \hyperlink{3.5}{(3.5)}. 

\begin{lem}
\hypertarget{7.6}{Let} \(G=Q_1\circ \cdots \circ Q_k\) be a central product. If \(V\) is an absolutely irreducible \(FG\)-module and \(V_i\subset V\) is an irreducible \(FQ_i\)-module, for all \(1\leq i\leq k\), then \(V\) is \(FG\)-isomorphic to \(V_1\otimes \cdots \otimes V_k\) and the \(V_i\) are absolutely irreducible.
\end{lem}

\begin{proof}
It will suffice to prove the statement for a direct product, since in view of \hyperlink{0.3.8}{(0.3.8)}, the result then also holds for a central product. Furthermore, it suffices to prove for the case when \(k=2\), since the general result follows by simple induction on \(k\). 

Let \(G=Q_1\times Q_2\) and \(V = M_1 \oplus \cdots \oplus M_d\) be the decomposition of \(V\) into its \(FQ_1\)-homogeneous components. By \hyperlink{2.6}{(2.6.iii)}, \(C_{G}(Q_1)Q_1\) stabilises each \(M_j\) and therefore (since \(Q_2\leq C_{GL(V)}(Q_1)\)), we know that \(G\) also stabilises this decomposition, but by the irreducibility of \(V\) as a \(FG\)-module, we must then have that \(V = M_1\). Thus \(V\) is \(FQ_1\)-isomorphic to \(V_1^{\oplus\,m}\), for some positive integer \(m\). The same argument, replacing \(Q_1\) with \(Q_2\), shows that \(V\) is \(FQ_2\)-isomorphic to \(V_2^{\oplus\,l}\), for some positive integer \(l\).

Define \(E=End_{FQ_1}(V_1)\). By \hyperlink{3.5}{(3.5)} and \hyperlink{3.6}{(3.6)}, \(C_{GL(V)}(Q_1)\cong GL_m(E)\) and \(Z(C_{GL(V)}(Q_1))\cong E^{\times}\). Following our convention, we will be referring to this subgroup of \(GL(V,F)\) as \(E^{\times}\). Since \(E^{\times}\) commutes with \(Q_1\) and \(C_{GL(V)}(Q_1)\), it commutes with \(G\), but by \hyperlink{0.3.11}{(0.3.11)}, \(C_{GL(V)}(G)=F^{\times}\). Thus \(E=F\) and again by \hyperlink{0.3.11}{(0.3.11)}, \(V_1\) is an absolutely irreducible \(FQ_1\)-module. The same argument shows \(V_2\) is an absolutely irreducible \(FQ_2\)-module.

As discussed in the proof of \hyperlink{3.5}{(3.5)}, if \(\alpha_i:V_1\longrightarrow V_i\) is a \(FQ_1\)-isomorphism, then \(\{\alpha_i\,|\, 1\leq i\leq m\}\) is a \(F\)-basis for the vector space \(A:=Hom_{FQ_1}(V_1,V)\cong E^m\). And if \(\{v_1,\:...\:,v_t\}\) is a \(F\)-basis of \(V_1\), then \(\{\alpha_i(v_j)\,|\, 1\leq i \leq m,\:1\leq j \leq t\}\) is a \(F\)-basis of \(V\). Thus, we can define a \(F\)-isomorphism \(\varphi: A\otimes V_1\longrightarrow V\) sending basis vectors \(\alpha_i \otimes v_j\mapsto \alpha_i(v_j)\).

Making use of the isomorphisms \(C_{GL(V)}(Q_1)\cong GL_m(E)\cong GL(A,E)\), there is a natural action of \(Q_1C_{GL(V)}(Q_1)\) on \(V_1\otimes A\), defined by \(q_1q_2\cdot (v\otimes w)= q_1 v\otimes q_2 w\). Furthermore, observing that \(G\leq Q_1C_{GL(V)}(Q_1)\), this action makes \(\varphi\) a \(FG\)-isomorphism. Indeed, if \(g=q_1q_2\in G\), then: 
\[\varphi(q_1q_2\cdot \alpha_i(v_j))=\varphi(q_2\cdot \alpha_i(q_1 v_j))=q_1 v_j\otimes q_2 \alpha_i = q_1q_2\cdot \varphi(\alpha_i(v_j))\]

Since \(V\) is an irreducible \(FG\)-module, this action of \(G\) on \( V_1\otimes A\) is also irreducible. Therefore \(A\) is an irreducible \(FQ_2\)-submodule of \(V\) and so by our previous observation, it must be \(FQ_2\)-isomorphic \(V_2\). Thus \(V\) is \(FG\)-isomorphic to \(V_1\otimes V_2\).
\end{proof}

\section{Aschbacher Class \(\mathcal{C}_8\) - Classical Form Stabilisers}

Before we state the main definition of this section, we prove a result that will enable the \(\mathcal{C}_8\) class to be emptied of any groups defined on a \(2\)-dimensional vector space.

\begin{pro}
\hypertarget{8.1}{Let} \(V\) be a \(2\)-dimensional \(\mathbb{F}_q\)-vector space, where \(q\) is odd. If \(Q\) is an orthogonal form of plus or minus type, then \(\Delta(V,Q)\) is contained in a member of \(\mathcal{C}_2\) or \(\mathcal{C}_3\).
\end{pro}

\begin{proof}
Recall that since \(q\) is odd, \(\Delta(V,Q)=\Delta(V,f_Q)\). First we consider the case when \(Q\) is of plus-type. By \hyperlink{0.1.13}{(0.1.13.iv)}, \(V\) admits a basis \(\{x,y\}\) such that \(f_Q(x,x)=f_Q(y,y)=0\) and \(f_Q(x,y)=1\). For all \(\lambda_1,\lambda_2\in F^{\times}\), we observe that \(f_Q(\lambda_1x+\lambda_2 y,\lambda_1x+\lambda_2 y)=2\lambda_1\lambda_2\), thus the scalar multiples of \(x\) and \(y\) are the only non-zero vectors that \(Q\) sends to zero. Therefore any similarity of \(Q\) permutes the subspaces \(\langle x \rangle\) and \(\langle y \rangle\), hence \(\Delta(V,Q)\leq N_{GL(V)}(\{\langle x \rangle,\langle y \rangle\})\in \mathcal{C}_2\).   

Next, suppose \(Q\) is of minus-type. By \hyperlink{0.2.10}{(0.2.10.v)}, \(S=S(V,Q)\) cannot be contained in the scalars and then, since \(V\) is \(2\)-dimensional, \(S\) is irreducible on \(V\). By the same argument of \hyperlink{3.4}{(3.4)}, \(E=End_{FS}(V)\) is a field containing \(F\), but by \hyperlink{0.3.12}{(0.3.12)}, \(S\) is not absolutely irreducible and therefore \(E\neq F\), by \hyperlink{0.3.11}{(0.3.11)}. Now \(V\) can be seen as vector space over \(E\), where scalar multiplication is just the action of the map. Since \(V\) is a \(F\)-vector space of dimsion \(2\) and \(E\) is a \(F\)-vector space of dimension greater than one, the \(E\)-dimension of \(V\) is \(1\) and thus \(E\cong \mathbb{F}_{q^2}\). 

So \(E^{\times}\) is a cyclic subgroup of order \(q^2-1\) and \(S\leq E^{\times}\) is of order \(q+1\), thus \(N_{GL(V)}(S)=N_{GL(V)}(E^{\times})\) by \hyperlink{0.2.17}{(0.2.17)}, and since \(S\) is characteristic in \(I=I(V,Q)\), any element normalising \(I\) will also normalise \(S\). Therefore \(\Delta(V,Q)=N_{GL(V)}(I)\leq N_{GL(V)}(S)=N_{GL(V)}(E^{\times})\in \mathcal{C}_3\), where the first equality holds by \hyperlink{0.3.21}{(0.3.21)}. 
\end{proof}

We are now ready to state the definition of a member of \(\mathcal{C}_8\).

\begin{defn}
A subgroup \(G\leq GL(V)\) is a member of \(\mathcal{C}_8\) if \(G= \Delta(V,f)\) where one of the following hold.

\begin{enumerate}[label=(\roman*),ref=(\roman*)]
 
\item The form \(f\) is unitary, \(q\) is a square and \(n\geq 3\). Such groups are isomorphic to \(GU_n(q^{1/2})\circ C_{q-1}\).
\item The form \(f\) is symplectic, \(n\geq 4\) and even. Such groups are isomorphic to \(GSp_n(q)\).
\item The form \(f\) is non-degenerate symmetric bilinear, \(q\) is odd and \(n\geq 3\). Such groups are isomorphic to \(GO^{\pm}_n(q)\).
\end{enumerate}
\end{defn}

\begin{rem}
\hypertarget{8.3}{The} conditions on \(q\) and \(n\) in our definition are either to ensure the group is well-defined or to avoid overlap with other Aschbacher classes. Further explanation is given below: 

\begin{enumerate}[label=(\roman*),ref=(\roman*)]
 
\item If \(f\) is unitary, then \(q\) must be a square for the form to exist on \(V\). The condition that \(n\geq 3\) is due to the isomorphism \hyperlink{0.2.10}{(0.2.10.iii)}, from which we deduce that if \(n=2\), then \(G\) is contained in \(\mathcal{C}_5\).

\item If \(f\) is symplectic, \(n\) must be even for the form to exist on \(V\). The condition that \(n\geq 4\) is due to the isomorphism \hyperlink{0.2.10}{(0.2.10.iii)}, from which we deduce that if \(n=2\), then \(G\geq SL(V)\); a case that is excluded by the statement of our main theorem.

\item For part (iii), we first note that the orthogonal groups were defined with the quadratic form \(Q\), not the associated symmetric bilinear form \(f\), thus it is primarily the quadratic form that we are interested in. However, if we have a quadratic form \(Q\), and \(q\) is even, then \(\Delta(V,Q)\leq \Delta(V,f')\), where \(f'\) is a symplectic form (as noted in \hyperlink{0.1.12}{(0.1.12)}) and therefore \(G\) is contained in a group covered by part (ii). Since \(q\) must be odd, \(\Delta(V,f)= \Delta(V,Q)\) by \hyperlink{0.1.12}{(0.1.12)}, and we can justifiably define \(G\) with respect to the associated symmetric bilinear form \(f\). We have chosen to define it this way to provide easy correspondence with results such as \hyperlink{0.3.19}{(0.3.19)} and \hyperlink{6.15}{(6.15)}. The condition that \(n\geq 3\) is explained by \hyperlink{8.1}{(8.1)}.

\item Aschbacher defines the \(\mathcal{C}_8\) class to consist of the group of \textit{semi-linear maps} that stabilise the forms mentioned in (i)-(iii) of our definition. This is an overgroup of \(\Delta(V,f)\); however, insofar as it relates to \(GL(V)\), both definitions yield the same members of \(C_8\).
\end{enumerate}

\end{rem}

\section{Proof of the Main Theorem}
 We are now ready to prove the main theorem of our paper. In the statement and proof of which, when we refer to a \textit{classical form} we use our regular definition \hyperlink{0.1.8}{(0.1.8)}, with the exclusion of two types of form on a \(2\)-dimensional space; a symplectic form and an orthogonal form over a field of even characteristic.

\begin{theorem*}(Aschbacher's Theorem for the General Linear Group)\\
Let \(F\) be a finite field and let \(V\) be a \(n\)-dimensional \(F\)-vector space, for some positive integer \(n\). If \(H\) is a subgroup of \(GL(V,F)\), not containing \(SL(V,F)\), then \(H\) is either contained in a member of one of the Aschbacher classes \(\mathcal{C}_1 - \mathcal{C}_8\) or the following hold.
\begin{enumerate}[label=(\roman*),ref=(\roman*)]
 \item \(H\) has a unique normal quasisimple subgroup \(L\).
 \item \(V\) is an absolutely irreducible \(FL\)-module that cannot be realised over any proper subfield of \(F\) and \(L\) does not fix any classical form on \(V\).
 \end{enumerate}
      
\end{theorem*}

\begin{proof}
Suppose that \(H\) is not contained in a member of any Aschbacher class. We may additionally assume without loss of generality that \(H\) contains \(F^{\times}\) (adopting our convention of identifying \(F^{\times}\) with the scalars of \(GL(V)\)). Indeed, in view of \hyperlink{0.2.21}{(0.2.21)} and the fact that \(HF^{\times}/H\) is abelian, the groups \(H\) and \(HF^{\times}\) have the same set of components.  

With these assumptions in place, we will arrive at conditions (i) and (ii) by proving a sequence of lemmas about \(H\), the first of which being: 
\begin{lemmat}
\(V\) \hypertarget{lemma 1}{is} an irreducible \(FH\)-module.
\end{lemmat} 

\begin{proof}
If \(V\) is reducible, then there exists a non-trivial subspace \(U\) stabilised by \(H\). Thus, \(H\leq N_{GL(V)}(U)\) is contained in a member of \(\mathcal{C}_1\), which is a contradiction. Therefore \(V\) is irreducible.
\end{proof}

For the next set of lemmas we will be considering the normal subgroups of \(H\). We define \(\mathcal{L}(H)= \{L\unlhd H | L\nleq F^{\times}\}\), which is clearly a non-empty set since \(H\) is a member. Let \(L\) be an arbitrary element of \(\mathcal{L}(H)\). 

\begin{lemmat}
\(V\) \hypertarget{lemma 2}{is} a homogeneous \(FL\)-module. 
\end{lemmat}
\begin{proof}
Let \(V = M_1 \oplus \cdots \oplus M_k\) be the decomposition of \(V\) into its \(FL\)-homogeneous components. By \hyperlink{2.6}{Clifford's theorem}, \(dim \:M_i = dim \:M_j\), for all \(\:1\leq i,j\leq k\), and \(H\) permutes the set \(\{M_1,...\,,M_k\}\). If \(k > 1\), then \(H\leq N_{GL(V)}(\{M_1 ,\:...\:, M_k\})\) is contained in a member of \(\mathcal{C}_2\), which is a contradiction. Therefore \(k = 1\) and \(V = M_1\) is \(FL\)-homogeneous.
\end{proof}

\begin{lemmat}
\hypertarget{lemma 3}{Each} irreducible \(FL\)-module is absolutely irreducible. 
\end{lemmat}
\begin{proof}
Let \(V = \bigoplus_{i=1}^d V_i\) be the decomposition of \(V\) into its irreducible \(FL\)-modules and define \(E_i=End_{FL}(V_i)\). By \hyperlink{3.5}{(3.5)}, we may identify \(E_i^{\times}\) with \(Z(C_{GL(V)}(L))\). 

We have shown in \hyperlink{3.4}{(3.4)} that \(E_i\) is a field containing \(F\). If \(E_i\neq F\), then there exists a field \(k\) such that  \(F\leq k \leq E_i\) and \(|k:F|\) is prime. By \hyperlink{3.6}{(3.6)}, \(N_{GL(V)}(L)\leq N_{GL(V)}(E_i^{\times})\) and thus \(H\) normalises \(E_i\). Since \(k\) is the unique subfield of its size in \(E_i\), \(H\) also normalises \(k\). So \(H\leq N_{GL(V)}(k)\) is contained in a member of \(\mathcal{C}_3\), which is a contradiction. Therefore \(E_i=F\) and by \hyperlink{0.3.11}{(0.3.11)}, \(V_i\) is an absolutely irreducible \(FL\)-module. 
\end{proof}

\begin{lemmat}
\(V\) \hypertarget{lemma 4}{is} an absolutely irreducible \(FL\)-module.
\end{lemmat}

\begin{proof}
We know that \(V\) is \(FL\)-isomorphic to \(V_1^{\,\oplus\, d}\), where \(d=\tfrac{n}{dim V_1}\). Suppose that \(d>1\) and let \(U\) be a \(F\)-vector space of dimension \(d\) on which \(L\) acts trivially. By \hyperlink{4.4}{(4.4)}, we can observe the \(FL\)-isomorphisms
\(V_1\otimes U \cong V_1^{\,\oplus\, d}\ \cong V\) and then by \hyperlink{4.3}{(4.3.ii)}, \(N_{GL(V)}(L) = GL(U)\otimes N_{GL(V_1)}(L)\) and hence we have the embeddings:
\[H\xhookrightarrow{      } GL(U)\otimes N_{GL(V_1)}(L) \xhookrightarrow{      } GL(U)\otimes GL(V_1)\]
So \(H\) is contained in a member of \(\mathcal{C}_4\), which is a contradiction. Therefore \(d=1\) and \(V=V_1\) is an absolutely irreducible \(FL\)-module.
\end{proof}

\begin{lemmat}
\(L\) \hypertarget{lemma 5}{cannot} be realised over any proper subfield of \(F\).
\end{lemmat}
\begin{proof}
Suppose that \(L\) can be realised over a subfield \(k\subset F\). The index of this subfield must divide \(n\) and by the remark \hyperlink{5.1}{(5.1.i)}, we can assume it is prime. In view of \hyperlink{lemma 4}{Lemma 4} above and \hyperlink{5.6}{(5.6)}, \(N_{GL(V)}(L)\leq N_{GL(V)}(W)F^{\times}\), where \(W\) is the \(k\)-span of some basis of \(V\). So \(H\leq N_{GL(V)}(W)F^{\times}\) is contained in a member of \(\mathcal{C}_5\), which is a contradiction. Therefore \(L\) cannot be realised over any proper subfield of \(F\).
\end{proof}

\begin{lemmat}
\(L\) \hypertarget{lemma 6}{is} not solvable.
\end{lemmat}
\begin{proof}
Suppose the set \(\{L\in \mathcal{L}(H)\:|\: L\text{ solvable}\}\) is non-empty and let \(L\) be a minimal element. Then \(L\) acts absolutely irreducibly by \hyperlink{lemma 4}{Lemma 4} above and so, in view of \hyperlink{0.3.11}{(0.3.11)}, \(Z(L)\leq F^{\times}\). By the definition of \(\mathcal{L}(H)\), \(L\neq Z(L)\), and by the minimality of \(L\), we can conclude that \(Z(L)\) is the unique maximal characteristic proper subgroup of \(L\). Therefore the derived subgroup \(L'\) is either equal to the whole group or contained in the centre, however the former is ruled out by the assumption that \(L\) is solvable. 

Since \(L'\) is central then, \(L\) is nilpotent and a direct product of its Sylow subgroups, each of which are characteristic in \(L\). At least one of these Sylow subgroups, say \(S\), must be non-central, but then it follows from our minimality supposition that \(S=L\). Thus, we have shown that \(L\) is a \(r\)-group for some prime \(r\). Furthermore, we have already seen that the unique maximal characteristic subgroup (and therefore every characteristic subgroup) is a subgroup of the scalars and hence is cyclic, so in accordance with definition \hyperlink{6.3}{(6.3)}, \(L\) is a symplectic-type \(r\)-group.

By \hyperlink{6.7}{(6.7)}, if \(r\) is odd then \(L=r^{1+2m}\), and if \(r=2\) then \(L=2_\pm^{1+2m}\) or \(C_4\circ 2^{1+2m}\). By \hyperlink{6.8}{(6.8)}, \(q=p^e\), where \(p\) is a prime distinct from \(r\) and by \hyperlink{6.17.iii}{(6.17.iii)}, \(e\) is the smallest integer such that \(p^e\equiv 1\: (mod\: |Z(L)|)\). If \(|Z(L)|>2\), \(e\) must be even, else by definition, \(N_{GL(V)}(L)\) (and therefore \(H\)) is contained in a member of \(\mathcal{C}_6\). Whereas if \(|Z(L)|=2\), then evidently \(e=1\). 

In either of these two possible cases, \(L\) satisfies one of the conditions in \hyperlink{6.14}{(6.14)} and therefore it must preserve a unitary, symplectic or non-degenerate symmetric bilinear form \(f\) on \(V\). But then by \hyperlink{6.15}{(6.15)}:
\[H\leq N_{GL(V)}(L) \leq \Delta(V,f) \in \mathcal{C}_8\]

\noindent yielding a contradiction. Therefore the set \(\{L\in \mathcal{L}(H)\:|\: L\text{ solvable}\}\) is empty. 
\end{proof}

\begin{lemmat}
\(H\) \hypertarget{lemma 7}{has} a unique normal quasisimple subgroup.  
\end{lemmat}
\begin{proof}
Let \(L\) be minimal in \(\mathcal{L}(H)\). Then \(L/Z(L)\) is a minimal normal subgroup of \(H/Z(L)\) and so by \hyperlink{0.2.15}{(0.2.15)}, we know that \(L/Z(L)=Y_1 \times \cdots \times Y_k\), where the \(Y_i\) are non-abelian simple subgroups of \(L/Z(L)\) that are conjugate in \(H/Z(L)\). Therefore, by \hyperlink{7.5}{(7.5)}, \(E(L)=Q_1 \circ \cdots \circ Q_k\), where the \(Q_i\) are quasisimple normal subgroups of \(L\) and conjugate in \(H\), which implies the \(Q_i\) are components of \(H\). Furthermore, since \(E(L)\) is non-central and characteristic in \(L\), by the minimality condition, we must have that \(L=E(L)\). Thus \(L\) is a product of components of \(H\).   

By \hyperlink{0.2.22}{(0.2.22)}, \(F^{\times}\) and all other components of \(H\) commute with \(L\), but by \hyperlink{lemma 4}{Lemma 4} and \hyperlink{0.3.11}{(0.3.11)}, \(C_H(L)\leq C_{GL(V)}(L)=F^{\times}\). Therefore, \(H\) can have no components other than those in \(L\). In other words, \(E(H)=L\).

Now let \(L=Q_1\circ \cdots \circ Q_k\) and suppose \(k>1\). By \hyperlink{7.6}{(7.6)}, \(L\) stabilises \(V=V_1\otimes \cdots \otimes V_k\), where \(V_i\) is an absolutely irreducible \(FQ_i\)-module, for each \(1\leq i\leq k\). Since \(H\) acts on \(E(H)\) by permuting the \(Q_i\), it follows that \(H\) also permutes \(C_{GL(V)}(Q_i)\). By defining \(C_j=\bigcap_{i\neq j} C_{GL(V)}(Q_i)\), it also follows that \(H\) permutes the \(C_j\) and so \(H\leq N_{GL(V)}(C_1\circ \cdots \circ C_k)\). But by \hyperlink{4.3}{(4.3.i)}, we can observe the isomorphisms \(C_j\cong \bigcap_{i\neq j} GL(\bigotimes_{t\neq i} V_t)\cong GL(V_j)\) and thus \(C_1\circ \cdots \circ C_k \cong GL(V_1)\otimes \cdots \otimes GL(V_k)\). Hence \(H\leq N_{GL(V)}(GL(V_1)\otimes \cdots \otimes GL(V_k))\) is contained in a member of \(\mathcal{C}_7\), which is a contradiction. Therefore \(k=1\) and \(E(H)=L=Q_1\) is quasisimple.
\end{proof}

\begin{lemmat}
\(E(H)\) \hypertarget{lemma 8}{does} not fix any classical form on \(V\).
\end{lemmat}

\begin{proof}
Suppose that \(L=E(H)\) does fix a classical form on \(V\). If \(n=2\), then the form is either unitary or orthogonal over an odd characteristic field (recalling our comment at the beginning of this section). If the former, then \(L\) is contained in a member of \(\mathcal{C}_5\) by \hyperlink{8.3}{(8.3.i)}. If the latter, then \(L\) is contained in a member of \(\mathcal{C}_2\) or \(\mathcal{C}_3\), by \hyperlink{8.1}{(8.1)}. If \(n>2\), then whatever type of form is fixed, \(L\) is contained in a member of \(\mathcal{C}_8\). In each case, we arrive at a contradiction. Therefore \(L\) does not fix a classical form on \(V\). 
\end{proof}

In view of \hyperlink{lemma 4}{Lemma 4}, \hyperlink{lemma 5}{Lemma 5}, \hyperlink{lemma 7}{Lemma 7} and \hyperlink{lemma 8}{Lemma 8}, we have established parts (i) and (ii) of the main theorem. 
\end{proof}

\clearpage

\bibliographystyle{abbrvnat}
\bibliography{citations2.bib}

\end{document}